\documentclass[a4paper,twoside,11pt,leqno]{article}
\usepackage{fullpage}
\usepackage{amsmath} \usepackage{amssymb} \usepackage{amsopn}
\usepackage{amsthm} \usepackage{amsfonts} \usepackage{mathrsfs}
\usepackage{mathabx}
\usepackage{ stmaryrd }

\usepackage{verbatim}

\usepackage[OT2,T1]{fontenc}

\usepackage[pdfpagelabels,pdftex]{hyperref}
\usepackage[dvips]{graphicx} \usepackage[usenames]{color}
\usepackage[all]{xy}

\usepackage{setspace}
\usepackage{amssymb}
\usepackage{euscript}
\usepackage{cite}
\usepackage[all]{xy}
\usepackage{tabularx}

\theoremstyle{plain}
\usepackage{amsfonts}
\usepackage{graphicx}
\usepackage{amsmath}

\hypersetup{
  pdftitle={},
  pdfauthor={},
  pdfsubject={},
  pdfkeywords={},
  colorlinks=false,
  breaklinks=true,
  bookmarksopen=true,
  bookmarksnumbered=true,
  pdfpagemode=UseOutlines,
  plainpages=false}

\newcommand{\bA}{{\mathbb A}}

\newcommand{\bF}{{\mathbb F}}
\newcommand{\bG}{{\mathbb G}}

\newcommand{\bP}{{\mathbb P}}
\newcommand{\bQ}{{\mathbb Q}}
\newcommand{\bR}{{\mathbb R}}

\newcommand{\bZ}{{\mathbb Z}}

\newcommand{\cA}{{\mathscr A}}
\newcommand{\cB}{{\mathscr B}}

\newcommand{\cF}{{\mathscr F}}

\newcommand{\cZ}{{\mathscr Z}}

\newcommand{\caO}{{\mathcal O}}

\newcommand{\fM}{{\mathfrak M}}

\newcommand{\fp}{{\mathfrak p}}

\DeclareSymbolFont{cyrletters}{OT2}{wncyr}{m}{n}
\DeclareMathSymbol{\Sha}{\mathalpha}{cyrletters}{"58}

\DeclareMathOperator{\pr}{pr}

\DeclareMathOperator{\Hom}{Hom}

\DeclareMathOperator{\GL}{GL}
\DeclareMathOperator{\SL}{SL}

\newcommand{\tr}{{\rm tr}}

\newcommand{\matzz}[4]{\left(
\begin{array}{cc} #1 & #2 \\ #3 & #4 \end{array} \right)}

\newcommand{\inv}{{\rm inv}}
\DeclareMathOperator{\Frob}{Frob}

\DeclareMathOperator{\R}{R}

\DeclareMathOperator{\Gal}{G}

\newcommand{\nr}{{\rm nr}}

\DeclareMathOperator{\Regul}{Reg}

\makeatletter
\newtheorem*{rep@theorem}{\rep@title}
\newcommand{\newreptheorem}[2]{%
\newenvironment{rep#1}[1]{%
 \def\rep@title{#2 \ref{##1}}%
 \begin{rep@theorem}}%
 {\end{rep@theorem}}}
\makeatother

\newtheorem{thm}{Theorem}[section]

\newtheorem{prop}[thm]{Proposition}

\newtheorem{cor}[thm]{Corollary}

\newtheorem{lm}[thm]{Lemma}

\newtheorem{sublm}[thm]{Sublemma}

\newreptheorem{thm}{Theorem}
\newreptheorem{prop}{Proposition}
\newreptheorem{cor}{Corollary}

\theoremstyle{definition}
\newtheorem{Def}[thm]{Definition}

\newtheorem{rem}[thm]{Remark}

\newenvironment{pro*}[1][Proof]{{\it{#1:}} }{}

\newcommand\Indd{\mathop{ \rm Ind}}

\newcommand\cIndd{\mathop{ \rm c-Ind}}

\newcommand\bk{\mathop{ \bar{k} }}

\newcommand\rar{ \rightarrow }
\newcommand\tar{ \twoheadrightarrow }
\newcommand\har{ \hookrightarrow }
\newcommand\Rar{ \Rightarrow }
\newcommand\LRar{ \Leftrightarrow }
\newcommand\longrar{\longrightarrow}

\newcommand\Res{\mathop{ \rm Res}}

\newcommand\charac{\mathop{ \rm char}}

\newcommand{\sm}{{\,\smallsetminus\,}}
\newcommand\N{\rm N}
\newcommand\Tr{\rm Tr}

\newcommand\coh{\rm H}

\newcommand\Ind{\mathop{ \rm Ind}}

\newcounter{absatzcounter}[section]
\numberwithin{equation}{section}

\begin{document}

\title{Affine Deligne-Lusztig varieties of higher level and the local Langlands correspondence for $GL_2$}
\author{A. Ivanov\footnote{The author was partially supported by ERC starting grant 277889 ''Moduli spaces of local $G$-shtukas''} \footnote{email: ivanov@ma.tum.de}}
\maketitle

\abstract{In the present article we define coverings of affine Deligne-Lusztig varieties attached to a connected reductive group over a local field of characteristic $p > 0$. In the case of $\GL_2$, the unramified part of the local Langlands correspondence is realized in the $\ell$-adic cohomology of these varieties. We show this by giving a detailed comparison with the realization of local Langlands via cuspidal types by Bushnell-Henniart. All proofs are purely local.}

\section{Introduction}

The classical Deligne-Lusztig theory aims for a geometric construction of representations of finite groups of Lie-type. In \cite{DL}, Deligne and Lusztig constructed the so-called Deligne-Lusztig varieties attached to a connected reductive group over a finite field and could show that any irreducible representation of the group of $\bF_q$-valued points occurs in the $\ell$-adic cohomology of these varieties. Since then one was trying to find similar constructions in the affine setting, aiming for a geometric realization of the local Langlands correspondence. However, usual geometric realizations of local Langlands make use of $p$-adic methods, formal schemes and adic spaces, also using the global theory. In the present article we introduce a very natural affine analog of Deligne-Lusztig varieties of arbitrary level attached to a connected reductive group over a local field $F$ of positive characteristic. Using these varieties we realize the unramified part of the local Langlands correspondence for $\GL_2$ 
over $F$ using only schemes over $\bF_q$ and purely local methods. Moreover, we will give a detailed comparison of our construction with the theory of cuspidal types of Bushnell-Henniart \cite{BH} and on the 'algebraic' side we will show an improvement of the Intertwining theorem \cite{BH} 15.1.

To begin with, let $k$ be a finite field with $q$ elements, $\bar{k}$ its algebraic closure and let $\sigma$ denote the Frobenius automorphism $x \mapsto x^q$ of $\bar{k}$. Let $F = k((t))$ resp. $L = \bar{k}((t))$ be the fields of Laurent series over $k$ resp. $\bar{k}$ and $\caO_F = k\llbracket t \rrbracket, \caO_L = \bar{k}\llbracket t \rrbracket$ their rings of integers. Let $\fp_L \subseteq \caO_L$ denote the maximal ideal. We extend $\sigma$ to an automorphism of $L$ by setting $\sigma(\sum_n a_nt^n) = \sum_n \sigma(a_n)t^n$.

Let $G$ be a connected reductive group over $F$ and let $\bG$ be a smooth model of $G$ over $\caO_F$. It is a central problem to realize smooth representations of the locally compact group $G(F)$ in the $\ell$-adic cohomology of certain schemes (or formal schemes, or adic spaces, ...) over $k$ (where $\ell$ is prime to $\charac(k)$). Usually such schemes come up with an action of $G(F) \times T(F)$, where $T$ is some maximal torus of $G$ and as a consequence the representations of $G(F)$ occurring in their $\ell$-adic cohomology are parametrized by characters of $T(F)$, lying in sufficiently general position. After the fundamental work of Deligne and Lusztig \cite{DL}, which followed the pioneering example of Drinfeld concerning $\SL_2(k)$, and deals with representations of the finite group $\bG(k)$, many generalizations of their ideas aiming a construction of representations of $\bG(\caO_F/t^r)$ for $r \geq 2$ resp. of $G(F)$ were made. We give some examples. In \cite{Lu} Lusztig suggested such construction 
(without proofs) and more recently he gave proofs in \cite{Lu2}. (A minor variation of) this construction was worked out for division algebras by Boyarchenko \cite{Bo} and Chan \cite{Ch} (see also \cite{BW}). A further closely related approach, was given by Stasinski in \cite{St}, who suggested a method to construct the so called extended Deligne-Lusztig varieties attached to $\bG(\caO_F/t^r)$. The advantages of our consturction are that it (i) has a quite simple definition in terms of the Bruhat-Tits building of $G$, (ii) establishes a direct link with affine Deligne-Lusztig varieties, which are well-studied in various contexts, which in particular allow to use the whole combinatoric machinery developped for their study, (iii) covers all levels (also level zero) simultaneously.

A starting point for our construction is Rapoport's definition of affine Deligne-Lusztig varieties in \cite{Ra} Definition 4.1. We recall this definition (in the Iwahori case). Let $\cB_L$ be the Bruhat-Tits building of the adjoint group $G_{L,ad}$. The Bruhat-Tits building of $G_{ad}$ over $F$ can be identified with the $\sigma$-invariant subset of $\cB_L$. Let $S$ be a maximal $L$-split torus in $G$, which is defined over $F$ (such a torus exists due to \cite{BT2} 5.1.12). Let $I \subseteq G(L)$ be the Iwahori subgroup attached to a $\sigma$-stable alcove in the apartment corresponding to $S$. Let $\cF$ be the affine flag manifold of $G$, seen as an ind-scheme over $k$. Its $\bk$-points can be identified with $G(L)/I$. Let $\tilde{W}$ denote the extended affine Weyl group of $G$ attached to $S$. The Bruhat decomposition of $G(L)$ induces the invariant position map

\[ \inv \colon \cF(\bk) \times \cF(\bk) \rar \tilde{W}. \]

\noindent For $w \in \tilde{W}$ and $b \in G(L)$ the affine Deligne-Lusztig variety attached to $w$ and $b$ is the locally closed subset 

\[ X_w(b) = \{ gI \in \cF \colon \inv(gI, b\sigma(g)I) = w \} \]

\noindent of $\cF$, which is given its reduced induced sub-Ind-scheme structure. Let $J_b$ be the $\sigma$-stabilizer of $b$, i.e., the algebraic group over $F$ defined by

\[ J_b(R) = \{ g \in G(R \otimes_F L) \colon g^{-1}b\sigma(g) = b \} \]

\noindent for any $F$-algebra $R$. Then $J_b(F)$ acts on $X_w(b)$. 

We sketch now the construction of natural covers of these varieties, which still admit an action by $J_b(F)$. The details are given in Section \ref{sec:sec_coverings_of_ADLV}. Let $\Phi = \Phi(G,S)$ be the relative root system. We see $0$ as the 'root' corresponding to the centralizer $T$ of $S$ in $G$ (as $G$ is quasi-split, this is a maximal torus). After choosing a $\sigma$-stable base point $x$ in $\cB_L$, with a concave function $f$ on $\Phi \cup \{0\}$ (for a definition cf. Section \ref{sec:AFM_and_covers}), one can associate a subgroup $G(L)_f \subseteq G(L)$. In \cite{Yu}, Yu defined a smooth model $\underline{G}_f$ of $G_L$ over $\caO_L$, such that $\underline{G}_f(\caO_L) = G(L)_f$. 
Assume that $G(L)_f \subseteq I$ and that $G(L)_f$ is $\sigma$-stable. Then $\underline{G}_f$ descends to a smooth group scheme over $\caO_F$. Further, $G(L)/G(L)_f$ is the set of $\bk$-points of an Ind-scheme $\cF^f$, which defines a natural cover of $\cF$ by \cite{PR} Theorem 1.4. Moreover, if $G(L)_f$ is normal in $I$, then $\cF^f \rar \cF$ is a (right) principal homogeneous space under $I/G(L)_f$. There is a map

\[ \inv^f \colon \cF^f(\bk) \times \cF^f(\bk) \rar D_{G,f}, \]

\noindent which covers the map $\inv$. Here $D_{G,f}$ is a set of representatives of double cosets of $G(L)_f$ in $G(L)$. For $w_f \in D_{G,f}$, $b \in G(L)$, we define the \emph{affine Deligne-Lusztig variety of level $f$} attached to $w_f$ and $b$ as the locally closed subset

\[ X^f_{w_f}(b) = \{ \bar{g} = gG(L)_f \in \cF^f(\bar{k}) \colon \inv^f(\bar{g},b\sigma(\bar{g})) = w_f  \}, \] 

\noindent endowed with its induced reduced sub-Ind-scheme structure  (in fact, this is a scheme locally of finite type over $k$). Assume $G(L)_f$ is normal in $I$. Then $I$ acts on $D_{G,f}$ by $\sigma$-conjugation $w_f \mapsto i^{-1}w_f \sigma(i)$, hence we can consider the stabilizer $I_{f,w_f} \subseteq I$ of $w_f$ under this action. It acts on $X^f_{w_f}(b)$ on the right and this action commutes with the left action of $J_b(F)$. Moreover this $I_{f,w_f}$-action can be extended to an action of $Z(F)I_{f,w_f}$, where $Z$ is the center of $G$. Thus we obtain the desired variety $X^f_{w_f}(b)$ with an action of $G(F) \times Z(F)I_{f,w_f}$. 

We study further properties of $I_{w_f,f}$ and $X^f_{w_f}(b)$ for general $G$ elsewhere. The rest of the paper is devoted to the detailed study of the case $G = \GL_2$. Now we explain our results in this case. As the levels indexed by concave functions are cofinal, we restrict attention to very special functions $f_m$ (cf. Sections \ref{sec:AFM_and_covers},\ref{sec:generalities_GL_2}) for integers $m \geq 0$ and write $I^m$ instead of $G(L)_{f_m}$, $X_{w_m}^m(1)$ instead of $X^{f_m}_{w_{f_m}}(b)$, etc.  We determine the varieties $X_{w_m}^m(1)$ and the attached representations of $G(F)$ and compare our results with the algebraic construction of the same representations in \cite{BH} using the theory of cuspidal types. We sketch our results here; for a precise treatment cf. Section \ref{sec:def_of_reps}. Let $E/F$ be the unramified extension degree $2$. If the image of $w_m$ in the finite Weyl group is non-trivial, then $Z(F)I_{m,w_m}$ has a natural quotient isomorphic to $E^{\ast}$, and the $Z(F)I_{m,w_m}$-
action in the $\ell$-adic cohomology of $X_{w_m}^m(1)$ factors through an $E^{\ast}$-action. In this way we obtain a $G(F)$-representation in the spaces $\coh_c^i(X_{w_m}^m(1), \overline{\bQ}_{\ell})[\chi]$, where $\chi$ goes through smooth $\overline{\bQ}_{\ell}^{\ast}$-valued characters of $E^{\ast}$. It turns out that if $\chi$ is minimal of level $m$, lies in sufficiently general position, then there is an integer $i_0$, such that $\coh_c^i(X_{w_m}^m(1), \overline{\bQ}_{\ell})[\chi] = 0$ for all $i \neq i_0$ and

\[ R_{\chi} =  \coh_c^{i_0}(X_{w_m}^m(1), \overline{\bQ}_{\ell})[\chi] \]

\noindent is an unramified irreducible cuspidal representation of $G(F)$, of level $m$ (we also define $R_{\chi}$ for $\chi$ non-minimal). Let $\bP_2^{\nr}(F)$ be the set of all isomorphism classes of admissible pairs over $F$ attached to $E/F$ (cf. \cite{BH} 18.2). Let $\cA_2^{\nr}(F)$ be the set of all isomorphism classes of unramfied irreducible cuspidal representations of $G(F)$. We defined a map

\begin{equation} \label{eq:R_realization_for_intro}
R \colon \bP_2^{\nr}(F) \rar \cA_2^{\nr}(F), \quad (E/F, \chi) \mapsto R_{\chi} .
\end{equation}

\noindent As a consequence of our trace computations in Sections \ref{sec:trace_I_preliminaries}-\ref{sec:trace_III_Hm}, we see that this map is injective (cf. Corollary \ref{cor:xi_chi_determines_chi}). Using the theory of cuspidal types and strata, Bushnell-Henniart attached to an admissible pair $(E/F,\chi)$ an irreducible cuspidal $G(F)$-representation $\pi_{\chi}$ (\cite{BH} \S19; we recall the construction briefly in Section \ref{sec:second_part_of_the_proof}). The tame parametrization theorem (\cite{BH} 20.2 Theorem) then shows that the map 

\begin{equation*} 
\bP_2^{\nr}(F) \stackrel{\sim}{\longrar} \cA_2^{\nr}(F), \quad (E/F, \chi) \mapsto \pi_{\chi}
\end{equation*}

\noindent is a bijection (also for even $q$). Here is our main result (which also works for even $q$).

\begin{repthm}{thm:main_thm} 
Let $(E/F,\chi)$ be an admissible pair. The representation $R_{\chi}$ is irreducible cuspidal, unramified, has level $\ell(\chi)$ and central character $\chi|_{F^{\ast}}$. Moreover, $R_{\chi}$ is isomorphic to $\pi_{\chi}$. In particular, the map \eqref{eq:R_realization_for_intro} is a bijection.
\end{repthm}

The proof is purely local. Two ideas in the proof follow \cite{Bo},\cite{BW}: it is Boyarchenko's trace formula (cf. Lemma \ref{lm:Boyarchenko_trace_formula}) and maximality of certain closed subvarieties of $X_{w_m}^m(1)$ (note that $X_{w_m}^m(1)$ itself is not maximal due to the presence of a 'level $0$ part'). The rest of the proof is independent of \cite{Bo},\cite{BW}.

Finally, we remark that for $G = \GL_2$ and $b$ superbasic, $J_b(F) = D^{\ast}$ for $D$ a quaternion algebra over $F$ and the varieties $X_{x_m}^m(b)$ seem to be very close (but unequal) to the varieties studied by Chan in \cite{Ch} (cf. Section \ref{sec:superbasic_case}).

\subsection*{Outline of the paper}

In Section \ref{sec:sec_coverings_of_ADLV} we define affine Deligne-Lusztig varieties for a connected reductive group $G$ of level attached to a concave function on the roots. In Section \ref{sec:Computations_for_GL2} we compute these varieties for $G = \GL_2$, $b = 1$ and determine their $\ell$-adic cohomology. In Section \ref{sec:def_of_reps} we recall the setup and state our main result for $\GL_2$, Theorem \ref{thm:main_thm}. After performing necessary trace calculations in Sections \ref{sec:trace_I_preliminaries}-\ref{sec:trace_III_Hm}, we compare our construction with that in \cite{BH} in Sections \ref{sec:first_part_of_the_proof}-\ref{sec:second_part_of_the_proof}, and finish the proof of Theorem \ref{thm:main_thm}.

\subsection*{Acknowledgments}

The author is especially grateful to Eva Viehmann, Christian Liedtke and Stephan Neupert for very helpful comments discussions concerning this work. Also he is grateful to Paul Hamacher, Bernhard Werner and other people for helpful remarks and interesting discussions concerning this work. Further, the author wants to thank Michael Rapoport and Ulrich G\"ortz for patiently introducing him to the field of affine Deligne-Lusztig varieties some years ago.


\section{Coverings of affine Deligne-Lusztig varieties}\label{sec:sec_coverings_of_ADLV}

The goal of this section is to define coverings of affine Deligne-Lusztig varieties. 


\subsection{Concave functions and smooth models} \label{sec:AFM_and_covers}

Let $G$ be a connected reductive group over $F$. As $\bar{k}$ is algebraically closed, $G_L$ is quasi-split over $L$. Let $S \subseteq G$ be a maximal $L$-split torus, which is defined over $F$. Let $T = \cZ_G(S)$ be the centralizer of $S$. As $G_L$ is quasi-split, $T$ is a maximal torus. Let $\Phi = \Phi(G_L,S_L)$ denote the relative root system. 
For $a \in \Phi$, write $U_a$ for the corresponding root subgroup and let $U_0 = T$. Let $\cB_L$ be the Bruhat-Tits building of $G_L$ and let $\cA_S$ be the apartment corresponding to $S_L$. We fix a $\sigma$-stable base alcove $\underline{a}$ contained in $\cA_S$ and let $x$ be one of its vertices. Then $x$ defines a filtration of $U_a(L)$ by subgroups $U_a(L)_{x,r}$ ($r \in \bR$) for $a \in \Phi$ (cf. \cite{BT} \S6.2). 

Moreover, choose an admissible schematic filtration on tori in the sense of Yu \cite{Yu} \S4. This gives a filtration $U_0(L)_{x,r} = T(L)_r$ on $T$. If $G$ satisfies condition (T) from \cite{Yu} 4.7.1, then this filtration is independent of the choice of the admissible filtration and coincide with the Moy-Prasad filtration on $T(L)$, cf. \cite{Yu} Lemma 4.7.4. Moreover, $G$ satisfies (T) if it is either simply connected or adjoint or split over a tamely ramified extension \cite{Yu} 8.1. We do not use this in the following.

Let $\tilde{\bR} = \bR \cup \{ r+ \colon r \in \bR\} \cup \{\infty\}$ be the monoid as in \cite{BT} 6.4.1. A function $f \colon \Phi \cup \{0\} \rar \tilde{\bR}$ is called \emph{concave} (\cite{BT} 6.4), if 
\[ \sum_{i = 1}^s f(a_i) \geq f(\sum_{i = 1}^s a_i). \]

\noindent Fix a concave function $f \colon \Phi \cup \{ 0 \} \rar \tilde{\bR}_{\geq 0} \sm \{\infty\}$. Let $G(L)_{x,f}$ be the subgroup of $G(L)$ generated by $U_a(L)_{x,f(a)}$, $a \in \Phi \cup \{ 0 \}$. By \cite{Yu} Theorem 8.3, there is a unique smooth model $\underline{G}_{x,f}$ of $G_L$ over $\caO_L$ such that $\underline{G}_{x,f}(\caO_L) = G(L)_{x,f}$. Moreover, if $G(L)_{x,f}$ is $\sigma$-stable, then $\underline{G}_{x,f}$ descends to a group scheme defined over $\caO_F$ (\cite{Yu} 9.1). We denote it again by $\underline{G}_{x,f}$.

Let $I \subseteq G(L)$ be the Iwahori subgroup associated with $\underline{a}$ and let $\Phi^+ \subseteq \Phi$ denote the set of positive roots determined by $\underline{a}$. Let $f_I$ be the concave function on $\Phi \cup \{ 0 \}$ defined by 

\[ f_I(a) = \begin{cases} 0 & \text{for $a \in \Phi^+ \cup \{0\}$ } \\ f_I(a) = 0+ & \text{for $a \in \Phi^-$}.\end{cases} \] 

\noindent Then $G(L)_{x,f_I} = I$ (cf. \cite{Yu} 7.3). For $m \geq 0$ let $f_m \colon \Phi \cup \{ 0 \} \rar \tilde{\bR}_{\geq 0} \sm \{ \infty \}$ be the concave function defined by

\[ f_m(a) = \begin{cases} m & \text{if $a \in \Phi^+$} \\ m^+ & \text{if $a \in \Phi^- \cup \{ 0 \}$.} \end{cases} \]

\noindent Write $I^m = G(L)_{x,f_m}$.

\begin{lm}
For $m \geq 0$, $I^m$ is normal in $I$ and $I^m$ is $\sigma$-stable. In particular, $I^m$ admits a unique smooth model $\underline{G}_{x,f_m}$. This model is already defined over $\caO_F$.
\end{lm}

\begin{proof}
$I$ (resp. $I^m$) is generated by $U_a(L)_{x,f_I(a)}$ (resp. $U_a(L)_{x,f_m(a)}$) for $a \in \Phi \cup \{ 0 \}$. To show normality, it is enough to show that for any roots $a,b \in \Phi \cup \{0\}$, the commutator $(U_a(L)_{x,f_I(a)}, U_b(L)_{x,f_m(b)})$ is contained in $I^m$. By \cite{BT} (6.2.1) V3 (we can treat $0$ as a root), $(U_a(L)_{x,f_I(a)}, U_b(L)_{x,f_m(b)})$ is contained in the subgroup generated by 
$U_{pa + qb}(L)_{pf_I(a) + qf_m(b)}$ for $p,q > 0$ such that $pa + qb \in \Phi \cup \{ 0 \}$. Now $qf_m(b) \geq m$, hence $U_{pa + qb}(L)_{pf_I(a) + qf_m(b)} \not\subseteq I^m$ can only happen if $f_I(a) = 0$, $f_m(b) = m$, $f_m(pa+qb) = m^+$. This is equivalent to $a \in \Phi^- \cup \{0\}$, $b \in \Phi^+$, $a+b \in \Phi^- \cup \{0\}$. This is impossible, hence $U_{pa + qb}(L)_{pf_I(a) + qf_m(b)} \subseteq I^m$ and the normality is shown. Further, $I^m$ is generated by $U_a(L)_{f_m(a)}$ for $a \in \Phi \cup \{0\}$, hence $\sigma(I^m)$ is generated by $\sigma(U_a(L)_{f_m(a)}) = U_{\sigma(a)}(L)_{f_m(a)}$. But as $I$ is $\sigma$-stable, we have $a \in \Phi^+ \LRar \sigma(a) \in \Phi^+$ and hence $f_m(a) = f_m(\sigma(a))$.
\end{proof}

Consider the loop group $LG$, which is the functor on the category of $k$-algebras,

\[ LG \colon R \mapsto G(R((t))). \]

\noindent Assume the concave function $f$ is such that $G(L)_{x,f}$ is $\sigma$-invariant. Let $L^+\underline{G}_{x,f}$ be the functor on the category of $k$-algebras defined by 

\[ L^+\underline{G}_{x,f} \colon R \mapsto \underline{G}_{x,f}(R[[t]]). \]

\noindent Then by \cite{PR} Theorem 1.4 the quotient of fpqc-sheaves

\[ \cF^f = LG/L^+\underline{G}_{x,f} \] 

\noindent is represented by an Ind-$k$-scheme of ind-finite type over $\bar{k}$ and its $\bar{k}$-points are $\cF^f(\bar{k}) =  G(L)/G(L)_{x,f}$. Moreover, if $g \leq f$ are two concave functions as above, then we have a natural projection $\cF^f \tar \cF^g$. We write $\cF = \cF^{f_I}$ for the affine flag manifold associated with $\underline{G}_{x,f_I}$, the smooth model of $I$ and $\cF^m = \cF^{f_m}$ for $m \geq 0$.


\subsection{Affine Deligne-Lusztig varieties and covers} \label{sec:def_of_cov_of_ADLV}

We keep the notations from Section \ref{sec:AFM_and_covers}. We fix a concave function $f \colon \Phi \cup \{ 0 \} \rar \tilde{\bR}_{\geq 0} \sm \{ \infty \}$, such that $f \geq f_I$, i.e., $G(L)_{x,f} \subseteq I$ and s.t. $G(L)_{x,f}$ is $\sigma$-invariant, i.e., $\underline{G}_{x,f}$ is defined over $\caO_F$. We write $I^f = G(L)_{x,f}$. There are natural $\sigma$-actions on $\cF(\bk),\cF^f(\bk)$, which are compatible with natural projections.

Let $\N_T$ be the normalizer of $T$ in $G$. Let $W = \N_T(L)/T(L)$ be the finite Weyl group associated with $S$ and $\tilde{W}$ the extended affine Weyl group. If $\Gamma$ denotes the absolute Galois group of $L$, then $\tilde{W}$ sits in the short exact sequence 

\[ 0 \rar X_{\ast}(T)_{\Gamma} \rar \tilde{W} \rar W \rar 0. \]

\noindent Then the Iwahori-Bruhat decomposition states that 

\[ G(L) = \coprod_{w \in \tilde{W}} I\dot{w}I, \]

\noindent where $\dot{w}$ is any lift of $w$ to $N(L)$. Consider now the set of double cosets 

\[ D_{G,f} = G(L)_{x,f}\backslash G(L)/G(L)_{x,f},\] 

\noindent equipped with the natural projection map $D_{G,f} \tar I\backslash G(L)/I \cong \tilde{W}$. If $m \geq 0$, we also write $D_{G,m}$ instead of $D_{G,f_m}$. At least for $w$ 'big' enough, the fiber $D_{G,f}(w)$ over a fixed $w \in \tilde{W}$ can be given the structure of a finite-dimensional affine variety over $\bar{k}$, by parametrizing it using subquotients of (finite) root subgroups. As this seems quite technical and as in this article we only need only the case $G = GL_2$ (cf. \eqref{eq:explicit_DNK_param_GL_2}), we omit the corresponding result in this article. We obtain a map, which covers the classical relative position map.

\begin{Def}
Define the map 
\[ \inv^f \colon \cF^f(\bk) \times \cF^f(\bk) \rar D_{G,f} \]
\noindent on $\bk$-points by $\inv^f(xG(L)_{x,f},yG(L)_{x,f}) = w_f$, where $w_f$ is the double $G(L)_{x,f}$-coset containing $x^{-1}y$.
\end{Def}

\noindent We come to our main definition.

\begin{Def} For $f \geq f_I$ concave, such that $I^f$ is $\sigma$-invariant, $b \in G(L)$, and $w_f \in D_{G,f}$ we define the \emph{affine Deligne-Lusztig variety of level $f$} associated with $b, w_f$ as

\[ X^f_{w_f}(b) = \{ \bar{g} = gG(L)_{x,f} \in \cF^f(\bar{k}) \colon \inv^f(\bar{g},b\sigma(\bar{g})) = w_f  \}, \] 

\noindent with its induced reduced sub-Ind-scheme structure.
\end{Def}

We write $X_{w_m}^m(b)$ instead of $X^{f_m}_{w_m}(b)$. As usual, $X^f_{w_f}(b)$ is equipped with two group actions. For $b \in G(L)$, let $J_b$ be the $\sigma$-stabilizer of $b$, i.e., the algebraic group over $F$ defined by

\[ J_b(R) = \{ g \in G(R \otimes_F L) \colon g^{-1}b\sigma(g) = b \} \]

\noindent for any $F$-algebra $R$. Then $J_b(F)$ acts on $X^f_{w_f}(b)$ for any $f$ and $w_f$. If $f \geq f^{\prime}$ and $w_f$ lies over $w_{f^{\prime}}$, then $X^f_{w_f}(b)$ lies over $X^{f^{\prime}}_{w_{f^{\prime}}}(b)$ and the $J_b(F)$-actions are compatible. 

To describe the second group action, assume additionally that $G(L)_{x,f}$ is normal in $I$. For $w \in \tilde{W}$, we have a left and a right $I/I^f$-action on $D_{G,f}(w)$ by multiplication. We obtain the (right) $I/I^f$-action on $D_{G,f}(w)$ by $(i,w_f) \mapsto i^{-1} w_f \sigma(i)$. 

\begin{lm}\label{lm:trivial_props} 
Assume $I^f$ is normal in $I$. Let $b \in G(L), w \in \tilde{W}$ and $w_f \in D_{G,f}(w)$. 
\begin{itemize}
\item[(i)]   $X^f_{w_f}(b)$ is locally of finite type over $k$.
\item[(ii)] For every $g \in G(L)$, the map $(h,xI^f) \mapsto (g^{-1}hg, g^{-1}xI^f)$ defines an isomorphism of pairs $(J_b(F),X^f_{w_f}(b)) \stackrel{\sim}{\longrar} (J_{g^{-1}b\sigma(g)}(F), X^f_{w_f}(g^{-1}b\sigma(g)))$.
\item[(iii)] For $i \in I$, the map $xI^f \mapsto xiI^f$ defines an isomorphism $X^f_{w_f}(b) \stackrel{\sim}{\longrar} X^f_{i^{-1}w_f\sigma(i)}(b)$.
\end{itemize}
\end{lm}

\begin{proof}
(ii) and (iii) are trivial computations. (i): The affine Deligne-Lusztig varieties $X_w(b)$ are locally of finite type, $\cF^f \tar \cF$ is a $I/I^f$-bundle and $I/I^f$ is of finite dimension over $\bar{k}$.
\end{proof}

By Lemma \ref{lm:trivial_props} (iii), the $\sigma$-stabilizer 
\[ I_{f,w_f} = \{ i \in I \colon i^{-1} w_f \sigma(i) = w_f \} \]

\noindent of $w_f \in D_{G,f}(w)$ in $I$ acts on $X^f_{w_f}(b)$ by right multiplication, and this action factors through an action of $I_{f,w_f}/I^f$. Let $Z$ denote the center of $G$. Note that $Z(F) \subseteq J_b(F)$, and that $J_b(F)$-action restricted to $Z(F)$ can also be seen as a right action, thus extending the right $I_{f,w_f}$-action on $X^f_{w_f}(b)$ to a right $Z(F)I_{f,w_f}$-action. If $m \geq 0$, we also write $I_{m,w_m}$ instead of $I_{f_m,w_m}$.


\section{Computations for $\GL_2$}\label{sec:Computations_for_GL2}

\emph{From now on until the end of the paper we set $G = \GL_2$.} In this section, we compute the associated varieties $X^m_{w_m}(1)$ and their $\ell$-adic cohomology. 

\subsection{Some notations and preliminaries}\label{sec:generalities_GL_2}

We fix the diagonal torus $T$ and the upper triangular Borel subgroup $B$ of $G$. We set $K = G(\caO_F)$ and fix the Iwahori subgroup $I$ and its subgroups $I^m$ for $m \geq 0$:

\[ I^m = \matzz{1+\fp_L^{m+1}}{\fp_L^m}{\fp_L^{m+1}}{1+\fp_L^{m+1}} \subsetneq I = \matzz{\caO_L^{\ast}}{\caO_L}{\fp_L}{\caO_L^{\ast}} \subseteq G(\caO_L). \]


\noindent Note that the groups $I^m$ coincide with those defined in Section \ref{sec:AFM_and_covers} with respect to the valuation on the root datum, which corresponds to the vertex of the Bruhat-Tits building of $G$ associated with the maximal compact subgroup $G(\caO_L)$. The maximal torus $T$ is split over $F$ and hence the filtration on it do not depend on the choice of an admissible schematic filtration. It is given by $T(L)_r = \matzz{1 + \fp^r}{0}{0}{1 + \fp^r}$. Let $W_a \subseteq \tilde{W}$ be the affine and the extended affine Weyl group of $G$.

The variety $X_w(1)$ is empty, unless $w = 1$ or $w \in W_a$ with odd length (cf. e.g. \cite{Iv} Lemma 2.4). 
The case $w = 1$ is not very interesting: $X_1(1)$ is a disjoint union of points and the cohomology of coverings of $X_{1}(1)$ contains the principal series representations of $G(F)$, as for classical Deligne-Lusztig varieties and as in \cite{Iv} in case of level $0$. Thus we restrict attention to elements of odd length in $W_a$. To simplify some computations, we fix once for all time a positive integer $n = 2k > 0$ and the elements 

\begin{equation} \label{eq:dot_w_and_v}
\dot{w} = \matzz{0}{t^{-n}}{-t^n}{0}, \quad \dot{v} = \matzz{t^k}{}{}{t^{-k}} \in \N_T(L) \subseteq G(L)
\end{equation}

\noindent and denote by $w$ (resp. $v$) the image of $\dot{w}$ (resp. $\dot{v}$) in $W_a$ and by $w_m$ the images of $\dot{w}$ in $D_{G,m}(w)$ (the elements with $n < 0$ can be obtained by conjugation; the elements with $n$ odd lead to similar results). Let $pr_m \colon \cF^m \rar \cF$ be the natural projection. 
We have the following parametrizations of $C_v^m = pr_m^{-1}(C_v)$ and $D_{G,m}(w)$.
For $m \geq 0$, let $\R_m$ denote the Weil restriction functor $\Res_{(k[t]/t^{m+1})/k}$ from $k[t]/t^{m+1}$-schemes to $k$-schemes. $C_v$ is parametrized by $R_{2k - 1} \bG_a \rar C_v$, $a \mapsto \matzz{1}{a}{}{1} v$, where $a = \sum_{i=0}^{2k-1} a_i t^i$. Then for $m \geq 0$, $C_v^m$ is parametrized by

\begin{eqnarray}\label{eq:explicit_Cvm_param_GL_2}
\psi_v^m \colon R_{2k-1} \bG_a \times \R_m \bG_m^2 \times \R_m \bG_a^2 &\stackrel{\sim}{\longrar}& C_v^m = IvI/I^m \nonumber \\ 
a, C,D,A,B &\mapsto& \matzz{1}{a}{}{1} \dot{v} \matzz{C}{}{}{D} \matzz{1}{A}{}{1} \matzz{1}{}{tB}{1}  I^m 
\end{eqnarray}

\noindent We write $a = \sum_{i=0}^{n-1} a_it^i$, $A = \sum_{i=0}^m A_it^i$ and $C = c_0(1 + \sum_{i=1}^{m} c_it^i)$. Moreover, for $m \leq n$, $D_{G,m}(w)$ is parametrized by 

\begin{eqnarray}\label{eq:explicit_DNK_param_GL_2}
\phi_w^m \colon \R_m \bG_m^2(\bar{k}) \times \R_{m-1} \bG_a^2(\bar{k}) &\stackrel{\sim}{\longrar}& D_{G,m}(w) = I^m\backslash IwI/I^m \nonumber \\
(C,D),(E,B) &\mapsto& I^m \matzz{1}{}{tE}{1} \dot{w} \matzz{C}{}{}{D} \matzz{1}{}{tB}{1} I^m.
\end{eqnarray}

\noindent The proof that $\psi_v^m$ resp. $\phi_w^m$ is an isomorphism of varieties resp. sets amounts to a simple computation. We omit the details.

Finally, we remark the existence of the following determinant maps. Let $x \in W_a$. There is a natural $k$-morphism of $k$-varieties:

\[ \det\nolimits^m \colon C_x^m = IxI/I^m  \rar \R_m \bG_m, \quad yI^m \mapsto \det(y) \mod t^{m+1}. \]

\noindent In the same way we have the $k$-morphism 
\[ \det\nolimits^m \colon D_{G,m}(w) \rar \R_m \bG_m, \quad I^m y I^m \mapsto \det(y) \mod t^{m+1}. \]


\subsection{The structure of $X^m_{w_m}(1)$}\label{sec:the_structure_of_Xm_wm1}


\begin{lm} \label{lm:I/Im-quot-structure} Let $m \geq 0$. There is a natural isomorphism
\[ I_{m,w_m}/I^m \stackrel{\sim}{\longrar} \left\{ \matzz{C}{A}{0}{D} \in G(\bk[t]/t^{m+1}) \colon \sigma^2(C) = C, D = \sigma(C) \right\}. \] 
\end{lm}

\begin{proof}
An easy computation using \eqref{eq:explicit_DNK_param_GL_2} shows the lemma.
\end{proof}

For $r > m$, let $\tau_m \colon \bk[t]/t^r \tar \bk[t]/t^{m+1}$ denote the reduction modulo $t^{m+1}$. Using coordinates from \eqref{eq:explicit_Cvm_param_GL_2}, let $S = \tau_m(\sigma(a) - a)$ and let $Y_v^m \subseteq C_v^m$ be the locally closed subset defined by 

\begin{eqnarray} \label{eq:Def_of_Y_m}
a_0 &\not\in& k \nonumber \\
B &=& 0  \\
\sigma(C) D^{-1} S^{-1} &=& 1 \nonumber \\ 
\sigma(D) C^{-1} S &=& 1 \nonumber 
\end{eqnarray}

\noindent Let $D_v \subseteq C_v$ be the open subset defined by the condition $a_0 \not\in k$. The composition $Y_v^m \rar C_v^m \rar C_v$ factors through $Y_v^m \rar D_v$. The natural $K$-action on $C_v^m$ by left multiplication restricts to an action on $Y_v^m$ (this will follow implicitly from the proof of Theorem \ref{thm:structure_of_level_m_coverings_b=1}). Moreover, Lemma \ref{lm:I/Im-quot-structure} implies that the natural right $I/I^m$-action on $C_v^m$ restricts to a right action of $I_{m,w_m}/I^m$ on $Y_v^m$.

\begin{thm}\label{thm:structure_of_level_m_coverings_b=1} 
Let $0 \leq m < n$. Let $w_m^{\prime} = \phi_w^m(C,D,E,B) \in D_{G,m}(w)$. Then $X_{w_m^{\prime}}^m(1)$ is non-empty if and only if one has $B = -\sigma(E)$. If this holds true, then $w_m^{\prime}$ is $I$-$\sigma$-conjugate to  $w_m = \phi_w^m(1,1,0,0)$ in $D_{G,m}(w)$ (that is $w_m^{\prime} = i^{-1} w_m \sigma(i)$ for some $i \in I$). In particular, $X_{w_m^{\prime}}^m(1) \cong X_{w_m}^m(1)$, compatible with appropriate group actions. Further, there is an isomorphism equivariant for the left $G(F)$- and right $(I/I^m)_{w_m}$-actions:

\[ X^m_{w_m}(1) \cong \coprod_{G(F)/K} Y_v^m. \]
\end{thm}

\begin{proof}
In \cite{Iv} it was shown that $X_w(1) = \coprod_{g \in G(F)/K} gD_v$ is the decomposition of $X_w(1)$ in connected components. As the natural projection $\cF^m \rar \cF$ restricts to a map \\ 
$\pr_m \colon X_{w_m^{\prime}}^m(1) \rar X_w(1)$, we have 

\[ X_{w_m^{\prime}}^m(1) \cong \coprod_{G(F)/K} \pr_m^{-1}(gD_v) = \coprod_{G(F)/K} g \pr_m^{-1}(D_v). \]

\noindent Thus it is enough to determine $pr_m^{-1}(D_v)$. One sees from Lemma \ref{lm:keycomputation}, that if $w_m^{\prime} = \phi_w^m(C,D,E,B)$ do not satisfy $B = -\sigma(E)$, then $pr_m^{-1}(D_v) = \emptyset$. On the other hand, if $w_m^{\prime}$ satisfies this, then $\sigma$-conjugating $w_m^{\prime}$ first by $\matzz{1}{}{B}{1} \in I$ and then by a diagonal $i = \matzz{i_1}{}{}{i_2} \in I$ such that $i_1^{-1}C\sigma(D)\sigma^2(i_1) = 1$ (such $i_1$ exists by Lang's theorem) and $i_2 = C\sigma(i_1)$, we deduce that $w_m^{\prime}$ is $I$-$\sigma$-conjugate to $w_m$. Thus by Lemma \ref{lm:trivial_props}(iii) we may assume $w_m^{\prime} = w_
m$. In this case Lemma \ref{lm:keycomputation} shows $pr_m^{-1}(D_v) = Y_v^m$, which finishes the proof.
\end{proof}

\begin{lm}[Key computation] \label{lm:keycomputation}
Let $0 \leq m < n$. Let $\dot{x}I^m = \psi_v^m(a, C, D, A, B) \in C_v^m$ such that $a_0 \not\in k$. Write $S = \tau_m(\sigma(a) - a)$. Then 

\[ \inv^m(\dot{x}I^m, \sigma(\dot{x})I^m) = \phi_w^m(\sigma(C) D^{-1} S^{-1}, \sigma(D) C^{-1} S, -B, \sigma(B)). \]
\end{lm}

\begin{proof} Let
\[ \dot{x} = \matzz{t^k}{t^{-k}a}{}{t^{-k}} \matzz{C}{}{}{D} \matzz{1}{A}{}{1} \matzz{1}{}{tB}{1} \in G(L).\] 

\noindent We have to compute the $(I^m,I^m)$-double coset of $\dot{x}^{-1}\sigma(\dot{x})$. By assumption $S$ is a unit and one computes (using $m<n$)

\[ \matzz{t^k}{t^{-k}a }{}{t^{-k}}^{-1} \sigma\matzz{t^k}{t^{-k}a }{}{t^{-k}} = \matzz{1}{t^{-2k}(\sigma(a) - a)}{}{1} \in I^m \matzz{S}{}{}{S^{-1}} \dot{w} I^m, \]


\noindent in $G(L)$. Thus by normality of $I^m$ in $I$, we obtain:

\[ 
\begin{split}
\dot{x}^{-1}\sigma(\dot{x}) \in I^m &\matzz{1}{}{-tB}{1} \matzz{1}{-A}{}{1} \matzz{C^{-1}}{}{}{D^{-1}} \cdot \matzz{S}{}{}{S^{-1}} \dot{w} \dots \\ 
\dots &\matzz{\sigma(C)}{}{}{\sigma(D)} \matzz{1}{\sigma(A)}{}{1} \matzz{1}{}{t\sigma(B)}{1} I^m.
\end{split}
\]

\noindent Then we can pull the term containing $-A$ to the right side of $\dot{w}$, without changing the other terms. The corresponding term, which then appear on the right side of $\dot{w}$ will lie in $I^m$, i.e., we can cancel it by normality of $I^m$ in $I$. The same can be done then with the term containing $\sigma(A)$, by pulling it to the left side of $\dot{w}$ and cancelling it.  Computing the remaining matrices together, we obtain:

\[ \dot{x}^{-1}\sigma(\dot{x}) \in I^m \matzz{1}{}{-tB}{1} \dot{w} \matzz{\sigma(C)D^{-1}S^{-1}}{}{}{C^{-1}\sigma(D)S} \matzz{1}{}{t\sigma(B)}{1} I^m. \]

\noindent This finishes the proof.
\end{proof}


\subsection{The structure of $Y_v^m$} \label{sec:structure_of_Yvm}

We keep notations from Sections \ref{sec:generalities_GL_2} and \ref{sec:the_structure_of_Xm_wm1}. Let $k_2/k$ denote the subextension of $\bar{k}/k$ of degree two. There is a natural surjection 

\[ I_{m,w_m} \tar T_{w,m} = \left\{ \matzz{C}{}{}{\sigma(C)} \colon C \in (k_2[t]/t^{m+1})^{\ast} \right\}. \]

\noindent Let $T_{w,m,0} = T_{w,m} \cap \SL_2(k_2[t]/t^{m+1})$ be the subgroup defined by the condition $C^{-1} = \sigma(C)$. Let $f \colon \bar{k} \rar \bar{k}$ denote the map $f(x) = x^q - x$. For $X \in \bk[t]/t^r$  we write $X = \sum_{i=0}^{r-1} X_i t^i$. We denote the affine space (over $\bk$) spanned by coordinates $X_0, \dots, X_{r-1}$ by $\bA^r(X_0, \dots, X_{r-1})$ resp. by $\bA^r(X)$.

\begin{prop}\label{prop:structure_of_Yvm_bla} Let $0 \leq m \leq n$. \mbox{}
\begin{itemize}
 \item[(i)] The variety $Y_v^m$ is isomorphic to the finite covering of $D_v \times \bA^m(A)$ given by 
 
\begin{equation}\label{eq:Yexplicit} 
\sigma^2(C)C^{-1} = \sigma(S)S^{-1}. 
\end{equation}

\noindent in $\bar{k}[t]/t^{m+1}$. It is a finite \'etale Galois covering with Galois group $T_{w,m}$.

\item[(ii)] The (set-theoretic) image of $\det^m \colon Y_v^m \rar \R_m \bG_m$ is the disjoint union of the $k$-rational points, which is as a set equal to $(k[t]/t^{m+1})^{\ast}$. Moreover, $\pi_0(Y_v^m) = (k[t]/t^{m+1})^{\ast}$. The map $\pi_0(Y_v^m) \tar \pi_0(Y_v^{m-1})$ induced by the projection corresponds to the reduction modulo $t^m$ map.

\item[(iii)] Let $Y_{v,0}^m$ be the connected component of $Y_v^m$ corresponding to $1 \in (k[t]/t^{m+1})^{\ast}$. Then $Y_{v,0}^m$ is (isomorphic to) a finite covering of $D_v \times \bA^m$ given by

\begin{equation}\label{eq:easy_eq_for_Yv0m} 
C\sigma(C) = S. 
\end{equation}

\noindent It is a connected finite \'etale Galois covering with Galois group $T_{w,m,0}$. Moreover, $Y_{v,0}^m \tar Y_{v,0}^{m-1} \times \bA^1(A_{m-1})$ is given by 

\begin{equation} \label{eq:defining_Ym_over_Ym-1}
c_m^q + c_m = \frac{f(a_m)}{f(a_0)} - \sum_{i=1}^{m-1} c_i^q c_{m-i}.
\end{equation}
\end{itemize}
\end{prop}

\begin{proof} Note that for a point $\psi_v^m(a,C,D,A,0) \in Y_v^m$, $S = \tau_m(a), C, D$ are units in $\bar{k}[t]/t^{m+1}$. From the last two equations in \eqref{eq:Def_of_Y_m}, we see that on $Y_v^m$, $D = \sigma(C)S^{-1}$ is uniquely determined by $C$ and $S$ and that $Y_v^m$ is indeed given by the equation \eqref{eq:Yexplicit}. Let us from now on proceed by induction on $m$. We see that $Y_v^0 \rar D_v$ is defined by $c_0^{q^2 - 1} = f(a_0)^{q-1}$, i.e., it is finite etale with Galois group isomorphic to $k_2^{\ast}$.  Clearly, $Y_v^m$ lies over $Y_v^{m-1} \times \bA^1(A_{m-1})$. Bring equation \eqref{eq:Yexplicit} to the form $\sigma^2(C)S = C\sigma(S)$. Expanding this expression with respect to $C = c_0(1 + \sum_{i=1}^m c_i t^i)$, $S = \sum_i f(a_i)t^i$, shows that $Y_v^m \rar Y_v^{m-1} \times \bA^1(A_{m-1})$ is defined by an equation of the form

\[c_m^{q^2} - c_m = p(a_0,\dots,a_m,c_0,\dots,c_{m-1}) \]


\noindent with $p$ some regular function on $Y_v^{m-1} \times \bA^1(A_{m-1})$. This is clearly a finite \'etale covering. Moreover, it is Galois and the Galois group is isomorphic to $k_2$, where $\lambda \in k_2$ acts by $c_m \mapsto c_m + \lambda$. By induction, $Y_v^m \rar D_v \times \bA^m(A)$ is also finite \'etale and has degree $(q^2 - 1)q^{2m}$. Equation \eqref{eq:Yexplicit} shows that the automorphism group of this covering contains $T_{w,m}$. Comparing the degrees we see that $Y_v^m \rar D_v \times \bA^m$ is Galois with Galois group $T_{w,m}$. This shows part (i). Let $\dot{x}I^m = \psi_v^m(a,C,D,A,0) \in Y_v^m$. As $Y_v^m \subseteq X_{w_m}^m(1)$, we obtain 
\[1 = \det\nolimits^m(\phi_w^m(1,1,0,0)) = \det\nolimits^m(I^m\dot{x}^{-1}\sigma(\dot{x})I^m) = (CD)^{-1}\sigma(CD). \]

\noindent It follows that $\det^m(\dot{x}I^m) = CD \in (k[t]/t^{m+1})^{\ast}$. This shows the claim about the image of $\det^m$. In particular, we obtain a map $\pi_0(Y_v^m) \rar (k[t]/t^{m+1})^{\ast}$. Its surjectivity follows using the action of $T_{w,m}$ on $Y_v^m$ and the fact that $\det \colon T_{w,m} \rar (k[t]/t^{m+1})^{\ast}$ is surjective. Let $Y_{v,0}^m$ be the preimage of $1$ under $\det^m \colon Y_v^m \rar \R_m \bG_m$. Then $Y_{v,0}^m$ is connected: this is a byproduct of Lemma \ref{lm:coh_of_YsmZ} (i) below. The compatibility of $\det^m$ with changing the level is immediate. Thus it remains to prove part (iii) of the proposition. Equation $CD = 1$, holding on $Y_{v,0}^m$, inserted into \eqref{eq:Def_of_Y_m} shows the first claim of (iii). The second statement of (iii) is clear from parts (i),(ii). Inserting $C = c_0(1 + \sum_{i=1}^m c_it^i)$, $S = \sum_i f(a_i)t^i$ into \eqref{eq:easy_eq_for_Yv0m} shows \eqref{eq:defining_Ym_over_Ym-1}. \qedhere

\end{proof}


\subsection{Cohomology of $Y_v^m$}\label{sec:coh_of_Yvm}

We keep the notations from Sections \ref{sec:generalities_GL_2}-\ref{sec:structure_of_Yvm}. Fix a prime $\ell \neq \charac(k)$. We are interested in the $\ell$-adic cohomology with compact support of $Y_v^m$. For a variety $X/\bF_q$, we write $\coh_c^i(X)$ instead of $\coh_c^i(X,\overline{\bQ_{\ell}})$. Set $h_c^i(X) = \dim_{\overline{\bQ}_{\ell}} \coh_c^i(X)$. Further, $\coh_c^i(X)(r)$ denotes the Tate twist. 
Set:

\begin{eqnarray*}
N_- &=& \{a_0,c_0 \in \bar{k} \colon a_0 \in k_2 \sm k, c_0^{q+1} = f(a_0) \} \subseteq \bk \times \bk, \\
\end{eqnarray*}

\noindent and let $C_+,C_-$ be affine curves over $\bF_q$ defined by

\[C_{\pm} \colon x^q \pm x = y^{q+1}. \]

\noindent We write $V_{\pm} = \coh^1(C_{\pm})$. Note that $V_+,V_-$ are isomorphic as abstract vector spaces (as $C_+, C_-$ are isomorphic over $\bar{k}$) and only differ by the action of $\Frob_q$ on them. One has $\dim_{\overline{\bQ_{\ell}}} V_{\pm} = q(q-1)$. 

\begin{thm}\label{thm:coh_of_Yvm}
Let $0 \leq m < n$. Then $\coh_c^i(Y_v^m) = \bigoplus_{(k[t]/t^{m+1})^{\ast}} \coh_c^i(Y_{v,0}^m)$. Let $d_0 = d_0(n,m) = 2(n-1) + 2m+1$. Then $\coh_c^i(Y_{v,0}^m) = 0$ if $i > d_0 + 1$ or $i < d_0 - m$ and 

\begin{eqnarray*}
\coh_c^{d_0 + 1}(Y_{v,0}^m) &\cong& \bQ_{\ell}(-(n+m))  \\ 
\coh_c^{d_0} &\cong& V_-(-(n+m-1))   \\ 
\coh_c^{d_0 - j}(Y_{v,0}^m) &\cong& \bigoplus_{N_-} \overline{\bQ}_{\ell}^{q^{2(j-1)}(q-1)} \qquad \text{for any $1 \leq j \leq m$.}
\end{eqnarray*}

\noindent For $1 \leq j \leq m$ the action of $\Frob_{q^2}$ on $\coh_c^{d_0 - j}(Y_{v,0}^m)$ is given as follows: it acts by permuting the blocks corresponding to elements of $N_-$ (by $(a_0,c_0) \mapsto (a_0, -c_0)$) and acts as multiplication with the scalar $(-1)^{d_0 - j} q^{d_0 - j}$ in each of these blocks. 
\end{thm}

\begin{rem}
We have chosen $d_0$ such that $\coh_c^{d_0 - j}(Y_v^m)$ corresponds to a $G(F)$-representation of level $j$ (cf. Definition \ref{Def:R_chi}).
\end{rem}

\begin{proof} The first statement of the theorem follows from Proposition \ref{prop:structure_of_Yvm_bla}. We need some further notation:
 
\begin{eqnarray*}
k_- = k_-(x) &=& \{x \in k_2 \colon x^q + x = 0 \} \subseteq k_2 \\ 
N_+ = N_+(x,y) &=& \{x,y \in \bar{k} \colon x \in k_2, y^{q+1} = x^q + x \} \subseteq \bk \times \bk .  
\end{eqnarray*}
 
\noindent Let $Y_{v,0}^{m,\prime}$ be the finite \'etale covering of the open subset $\{a_0 \not\in k \}$ of the $m+1$-dimensional affine space $\bA^{m+1}(a_0,\dots,a_m)$, which is defined by the same equations defining $Y_{v,0}^m$ (cf. \eqref{eq:defining_Ym_over_Ym-1}). There is a projection $Y_{v,0}^m \tar Y_{v,0}^{m,\prime}$ and the $I_{m,w_m}/I^m$-action on $Y_{v,0}^m$ induces a $T_{w,m}$-action on $Y_{v,0}^{m,\prime}$. We have $Y_{v,0}^m \cong Y_{v,0}^{m,\prime} \times \bA^{n - 1}(a_{m+1},\dots,a_{n-1}, A_0,\dots,A_{m-1})$ and hence $\coh_c^i(Y_{v,0}^m) = \coh_c^{i - 2(n-1)}(Y_{v,0}^{m,\prime})(-(n-1))$. For $m \geq 0$, let $Z^m \subseteq Y_{v,0}^{m,\prime}$ be the closed subscheme defined by the equation $a_0^{q^2} - a_0 = 0$. Note that $K$- and $T_{w,m}$-actions on $Y_{v,0}^{m,\prime}$ restrict to actions on $Z^m$ and that equation \eqref{eq:defining_Ym_over_Ym-1} defines $Z^m \subseteq Z^{m-1} \times \bA^1(c_m)$ as a covering of $Z^{m-1}$. For all $m \geq 1$, we make the following coordinate change: 
replace $a_m$ by $a_m^{\prime} = f(a_0)^{-\frac{
1}{q}}a_m - c_m$. For $a_0 \in k_2 \sm k$, one computes $f(a_0)^q = - f(a_0)$ and equation \eqref{eq:defining_Ym_over_Ym-1} simplifies over the locus $a_0 \in k_2 \sm k$ to 

\begin{equation} \label{eq:first_coord_change_for_Zm1}
a_m^{\prime,q} + a_m^{\prime} = \sum_{i=1}^{m-1}c_i^q c_{m-i}. 
\end{equation}

\noindent Make a second coordinate change: for all $m \geq 1$ replace $a_m^{\prime}$ by $\alpha_m = a_m^{\prime} - \sum_{i=1}^{\lfloor \frac{m-1}{2} \rfloor} c_i^q c_{m-i}$. This second coordinate change turns equation \eqref{eq:first_coord_change_for_Zm1} defining $Z^m$ over $Z^{m-1}$ into

\begin{equation} \label{eq:first_coord_change_for_Zm2}
\alpha_m^q + \alpha_m = \sum_{i=1}^{\lfloor \frac{m-1}{2} \rfloor} (c_i - c_i^{q^2}) c_{m-i}^q + \delta_m c_{m/2}^{q+1},
\end{equation}

\noindent where $\delta_m = 0$ if $m$ is odd and $\delta_m = 1$ if $m$ is even. All together, $Z^m$ is isomorphic to the locally closed subset of $\bA^{2m+2}(a_0,\alpha_1\dots,\alpha_m,c_0,c_1,\dots,c_m)$ defined by $a_0^{q^2} - a_0 = 0$, $a_0^q - a_0 \neq 0$, $c_0^{q+1} = a_0^q - a_0$ and the $m$ equations \eqref{eq:first_coord_change_for_Zm2} for $m^{\prime} = 1,2,\dots,m$. The first three of these equations and the equation \eqref{eq:first_coord_change_for_Zm2} for $m^{\prime} = 1$ obviously divide $Z^m$ into $N_- \times k_-(\alpha_1)$ components, which are all isomorphic, as one sees using $K$- and $T_{w,m}$-actions on $Z^m$. Thus $Z^m \cong \coprod_{N_- \times k_-(\alpha_1)} Z_0^m$, where $Z_0^m$ is the closed subvariety of $\bA^{2m-2}(\alpha_2,\dots, \alpha_m, c_1,\dots,c_m)$ defined by equations \eqref{eq:first_coord_change_for_Zm2} for $m^{\prime} = 2,\dots, m$.

\begin{lm}\label{lm:conn_cpts_of Z_m}
Let $m \geq 1$. Then $Z_0^m$ is connected, i.e., $\pi_0(Z^m) = N_- \times k_-(\alpha_1)$.
\end{lm}

\begin{proof}  We proceed by induction: for $m=0$, $Z_0^0$ is a point, thus connected. Let $m \geq 1$ and assume that $Z_0^{m-1}$ is connected. By Lemma \ref{lm:coh_of_YsmZ}(ii) below (this lemma is formulated for $\tilde{Z}_m$ instead of $Z^m \cong \tilde{Z}^{m-1} \times \bA^1(c_m)$ -- see below in the proof of the theorem), the fibers of $Z^m \rar Z^{m-1}$ (and hence also of $Z_0^m \rar Z_0^{m-1}$) over the open subset defined by $c_1^{q^2} - c_1 \neq 0$ are connected. Hence $Z_0^m$ is connected.
\end{proof}

We see that $Z_0^m$ is a connected \'etale covering of $\bA^m(c_1,\dots,c_m)$. Hence $\coh_c^i(Z^m) = 0$ for $i > 2m$ and $\coh_c^{2m}(Z^m) = \bigoplus_{N_- \times k_-(\alpha_1)} \bQ_{\ell}(-m)$. Consider now the decomposition in an open and a closed subset:
\begin{equation}\label{eq:MV_seq_ZYYZ} 
Y_{v,0}^{m,\prime} \sm Z^m \har Y_{v,0}^{m,\prime} \hookleftarrow Z^m. 
\end{equation}

\noindent Lemma \ref{lm:coh_of_YsmZ} shows that $\coh_c^i(Y_{v,0}^{m,\prime} \sm Z^m) = \coh_c^{i - 2m}(Y_{v,0}^{0,\prime} \sm Z^0)(-m)$ and $Y_{v,0}^{0,\prime} \sm Z^0$  can be identified with the open subset $C_- \sm N_-$ of the curve $C_-$ defined in the variables $a_0,c_0$. 

\begin{lm}\label{lm:delta_in_MV_non_triv}
In the long exact sequence for $\coh_c^{\ast}(\cdot)$ attached to \eqref{eq:MV_seq_ZYYZ} (cf. \cite{Mi} III \S1 Remark 1.30), the map 
\[ \delta_m \colon \coh_c^{2m}(Z^m) \rar \coh_c^{2m+1}(Y_{v,0}^{m,\prime} \sm Z^m) = \bigoplus_{N_-} \overline{\bQ}_{\ell}(-m) \oplus V_-(-m). \]
is surjective onto the first summand.
\end{lm}

\begin{proof}
By comparing the Frobenius-weights (which is possible due to Lemma \ref{lm:coh_of_YsmZ}) we see that the image is contained in the first summand. On the other hand, the natural projection $Y_{v,0}^{m,\prime} \rar Y_{v,0}^{m-1,\prime}$ induces a morphism between the corresponding long exact sequenceses for $\coh_c^{\ast}(\cdot)$, which induces a commutative diagram relating $\delta_m$ with $\delta_{m-1}$. Iterating this for all levels $\geq 1$, we obtain a commutative diagram:

\centerline{
\begin{xy}\label{diag:character_isos_diag}
\xymatrix{
\bigoplus_{N_- \times k_-} \overline{\bQ}_{\ell}(-m) \ar@{=}[r] \ar@{->>}[d] & \coh_c^{2m}(Z^m) \ar[r]^(.4){\delta_m} \ar@{=}[d] & \coh_c^{2m+1}(Y_{v,0}^{m,\prime} \sm Z^m) \ar@{->>}[d] & \\
\bigoplus_{N_-} \overline{\bQ}_{\ell} \ar@{=}[r] & \coh_c^0(Z^0) \ar[r]^(.4){\delta_0} & \coh_c^1(C_- \sm N_-) \ar@{=}[r] & \bigoplus_{N_-} \overline{\bQ}_{\ell} \oplus V_-
}
\end{xy}
}
\noindent This diagram shows the lemma.\qedhere
\end{proof}

The long exact sequence for $\coh_c^{\ast}(\cdot)$ and Lemma \ref{lm:delta_in_MV_non_triv} implies:

\begin{equation}\label{eq:coh_of_Yvm0_prime}
\coh_c^i(Y_{v,0}^{m,\prime}) = 
\begin{cases} 
\coh_c^i(Z^m) & \text{if $i < 2m$} \\
\bigoplus_{N_-} [\bigoplus_{k_-(\alpha_1)} \overline{\bQ_\ell}(-m)]^{\sum_{\alpha_1} = 0} & \text{if $i = 2m$} \\
V_-(-m) & \text{if $i = 2m + 1$} \\
\bQ_{\ell}(-m-1) & \text{if $i = 2m+2$} \\
0 & \text{if $i > 2m+2$,}
\end{cases}
\end{equation}

\noindent where $\sum_{\alpha_1} = 0$ means that we take the subspace of trace zero elements. Let $\tilde{Z}^{m-1}$ be the closed subspace of $\bA^{2m+1}(a_0,\alpha_1, \dots,\alpha_m, c_0, c_1,\dots,c_{m-1})$ given by the same equations as $Z^m$: $a_0^{q^2} - a_0 = 0$, $a_0^q - a_0 \neq 0$, $c_0^{q+1} = a_0^q - a_0$ and equations \eqref{eq:first_coord_change_for_Zm2} for $m^{\prime} = 1, \dots, m$. Let $H = \{ c_1^{q^2} - c_1 = 0 \}$ be a finite union of hyperplanes in the same affine space. Then $Z^m \cong \tilde{Z}^{m-1} \times \bA^1(c_m)$. For $m \geq 3$ Lemma \ref{lm:coh_of_YsmZ}(ii) shows (here and until \eqref{eq:dim_coh_Zm} we ignore Tate twists):

\[ \coh_c^i(\tilde{Z}^{m-1} \sm H) = \coh_c^{i - 2(m-2)}(\tilde{Z}^1 \sm H) = 
\begin{cases} 
0 & \text{if $i \leq 2m - 4$} \\ 
\bigoplus\limits_{N_- \times k_-(\alpha_1)} [(\bigoplus\limits_{N_+(\alpha_2,c_1)} \overline{\bQ}_{\ell}) \oplus V_+] & \text{if $i = 2m - 3$} \\
\bigoplus\limits_{N_- \times k_-(\alpha_1)} \overline{\bQ}_{\ell} & \text{if $i = 2m - 2$},
\end{cases} \] 

\noindent because $\tilde{Z}^1 \sm H \cong \coprod_{N_- \times k_-(\alpha_1)} (C_+ \sm N_+)$. Further, Lemma \ref{lm:induction_on_Z_tilde_Z} shows $\tilde{Z}^{m-1} \cap H = \coprod_{k_-(\alpha_1) \times N_+(\alpha_2,c_1)} Z^{m-2}_{(1)}$, where $Z^{m-2}_{(1)} \cong Z^{m-2}$ and the index $(1)$ indicates the shift in variables given by $\alpha_i \mapsto \alpha_{i+2}, c_i \mapsto c_{i+1}$ (for $i \geq 1$) and hence $\coh_c^i(\tilde{Z}^{m-1} \cap H) = \bigoplus_{k_-(\alpha_1) \times N_+(\alpha_2,c_1)} \coh_c^i( Z^{m-2}_{(1)} )$ and, in particular, the top cohomology group of $\tilde{Z}^{m-1} \cap H$ is in degree $2m - 4$ and is equal to 

\[ \coh_c^{2m - 4}(\tilde{Z}^{m-1} \cap H) = \bigoplus_{k_-(\alpha_1) \times N_+(\alpha_2,c_1)} \coh_c^{2m-4}(Z_{(1)}^{m-2}) = \bigoplus_{k_-(\alpha_1) \times N_+(\alpha_2,c_1)}\bigoplus_{N_- \times k_-(\alpha_3)} \overline{\bQ}_{\ell}. \]

\noindent as follows from Lemma \ref{lm:conn_cpts_of Z_m} (note the index shift $\alpha_1 \mapsto \alpha_3$). All these, the long exact sequence for $\coh_c^{\ast}(\cdot)$ attached to

\[ \tilde{Z}^{m-1} \sm H \har \tilde{Z}^{m-1} \hookleftarrow \tilde{Z}^{m-1} \cap H, \]

\noindent the analog of Lemma \ref{lm:delta_in_MV_non_triv} for this sequence and Lemma \ref{lm:induction_on_Z_tilde_Z} show that for $m \geq 3$ we have:

\begin{equation}\label{eq:coh_of_Zi_inductive}
\coh_c^i(Z^m)  = \coh_c^{i-2}(\tilde{Z}^{m-1}) = 
\begin{cases} 
\bigoplus\nolimits_{k_-(\alpha_1) \times N_+(\alpha_2,c_1)} \coh_c^{i-2}(Z^{m-2}_{(1)}) & \text{if $i < 2m - 2$} \\
\bigoplus\nolimits_{N_- \times k_-(\alpha_1) \times N_+(\alpha_2,c_1)} [\bigoplus_{k_-(\alpha_3)} \overline{\bQ}_{\ell}]^{\sum_{\alpha_3} = 0} & \text{if $i = 2m - 2$} \\
\bigoplus\nolimits_{N_- \times k_-(\alpha_1)} V_+ & \text{if $i = 2m - 1$} \\
\bigoplus\nolimits_{N_- \times k_-(\alpha_1)} \overline{\bQ}_{\ell} & \text{if $i = 2m$.}
\end{cases} 
\end{equation}

\noindent Note that $\sharp N_+ = q^3, \sharp k_- = q$. Hence for $m \geq 3$, we have $h_c^{2m}(Z^m) = (\sharp N_-)q$ and $h_c^{2m - j}(Z^m) = (\sharp N_-)q^{2j}(q - 1)$ for $j \in \{ 1,2 \}$. For $Z^1, Z^2$ one computes: $\coh_c^2(Z^1) = \bigoplus_{N_- \times k_-} \overline{\bQ}_{\ell}$ and $\coh_c^i(Z^1) = 0$ if $i \neq 2$ and 

\begin{equation}\label{eq:coh_of_Z2}
\coh_c^i(Z^2)  = \coh_c^{i-2}(\tilde{Z}^1) = 
\begin{cases} 
0 & \text{if $i \leq 2$ or $i \geq 5$} \\
\bigoplus\nolimits_{N_- \times k_-(\alpha_1) } V_+ & \text{if $i = 3$} \\
\bigoplus\nolimits_{N_- \times k_-(\alpha_1)} \overline{\bQ}_{\ell} & \text{if $i = 4$.}
\end{cases} 
\end{equation}

Let now $m \geq 3$. For $j > 0$, write $j = 2 \lfloor \frac{j - 1}{2} \rfloor + j^{\prime}$, where $j^{\prime} = 1$ if $j$ odd, $j^{\prime} = 2$ otherwise. Iterating \eqref{eq:coh_of_Zi_inductive} $\lfloor \frac{j - 1}{2} \rfloor$ times, we get for all $0 < j < m$:

\begin{equation} \label{eq:iterated_formula}
h_c^{2m-j}(Z^m) = q^{4 \lfloor \frac{j - 1}{2} \rfloor} h_c^{2(m - 2 \lfloor \frac{j - 1}{2} \rfloor ) - j^{\prime}}(Z^{m - 2\lfloor \frac{j - 1}{2} \rfloor}_{(\lfloor \frac{j - 1}{2} \rfloor)}  )  = (\sharp N_-)q^{2j}(q-1).
\end{equation}

\noindent where $Z^m_{(l)} \cong Z^m$ using the index shift as above $l$ times. Thus for all $m \geq 1$, $j > 0$: 

\begin{equation}\label{eq:dim_coh_Zm} 
h_c^{2m + 1 - j}(Z^m) = \begin{cases} (\sharp N_-) q & \text{if $j = 1$} \\ (\sharp N_-) q^{2(j-1)}(q - 1) & \text{if $1 < j \leq m$} \\ 0 & \text{otherwise.} \end{cases} 
\end{equation}

\noindent Combined with \eqref{eq:coh_of_Yvm0_prime}, this implies the dimension formula in the theorem. It remains to compute the Frobenius action. The Tate twists in the two top cohomology groups of $Y_{v,0}^m$ can be deduced easily by relating $Y_{v,0}^m$ with $Y_{v,0}^{m,\prime}$. To prove the claim about $\Frob_{q^2}$-action in degrees $i \leq d_0 - 1$, note that $\Frob_{q^2}$ acts on $N_-$ by $(a_0,c_0) \mapsto (a_0, -c_0)$. Further, let $Z^m_1$ be the subvariety of $\bA^{2m}(c_1,\dots,c_m,a_1^{\prime},\dots,a_m^{\prime})$ defined by $m$ equations \eqref{eq:first_coord_change_for_Zm1} for $m^{\prime} = 1, \dots, m$, i.e., $Z^m = N_- \times Z^m_1$, where $N_-$ is seen as a discrete variety. Lemma \ref{lm:Zm1_maximality} shows that $Z^m_1$ is a maximal variety over $\bF_{q^2}$ (for a definition cf. the paragraph preceding Lemma \ref{lm:Zm1_maximality}), i.e., $\Frob_{q^2}$ acts in $\coh_c^i(Z^m_1)$ by $(-1)^i(q^2)^{i/2} = (-q)^i$ for any $i \in \bZ$. Further we have for all $2 \leq j \leq m$:

\[ \coh_c^{d_0 - j}(Y_{v,0}^m) = \coh_c^{2m + 1 - j}(Y_{v,0}^{m,\prime})(-(n-1)) = \coh_c^{2m + 1 - j}(Z^m)(-(n-1)) \]

\noindent (note that for $j = 1$, this remains true if one replaces the second equality by an inclusion, cf. \eqref{eq:coh_of_Yvm0_prime}). This implies the last statement of the theorem. \qedhere

\end{proof}

\begin{lm} \label{lm:coh_of_YsmZ} With notations as in the proof of Theorem \ref{thm:coh_of_Yvm}, we have:
\begin{itemize}
\item[(i)] Let $m \geq 1$. The fibers of the natural projection $\pi \colon Y_{v,0}^{m,\prime} \sm Z^m \rar Y_{v,0}^{m-1,\prime} \sm Z^{m-1}$ are isomorphic to $\bA^1$. We have:
\[ \coh_c^i(Y_{v,0}^{m,\prime} \sm Z^m) = \coh^{i-2}(Y_{v,0}^{m-1,\prime} \sm Z^{m-1})(-1). \]
\item[(ii)] Let $m \geq 3$. The fibers of the natural projection $\tilde{Z}^{m-1} \sm H \rar \tilde{Z}^{m-2} \sm H$ are isomorphic to $\bA^1$. We have:
\[ \coh_c^i(\tilde{Z}^{m-1} \sm H) = \coh_c^{i-2}(\tilde{Z}^{m-2} \sm H)(-1). \]
\end{itemize}
\end{lm}

\begin{proof}[Proof of Lemma \ref{lm:coh_of_YsmZ}] Let us prove part (i). The scheme $Y_{v,0}^{m,\prime} \sm Z^m$ is the closed subspace of $(Y_{v,0}^{m-1,\prime} \sm Z^{m-1}) \times \bA^2(a_m,c_m)$ defined by the equation $\eqref{eq:defining_Ym_over_Ym-1}$. Letting $x$ be a point of $Y_{v,0}^{m-1,\prime} \sm Z^{m-1}$, we see that the fiber of $\pi$ over $x$ is given by the equation 
\[ c_m^q + c_m = f(a_0(x))^{-1} (a_m^q - a_m) + \lambda(x), \]
\noindent with $a_0(x) \not\in k_2$ the $a_0$-coordinate of $x$ and $\lambda(x) \in \bar{k}$ depending on $x$. Using the coordinate change $a_m^{\prime} = f(a_0(x))^{-\frac{1}{q}} a_m - c_m$, this equation can be rewritten as

\[ a_m^{\prime,q} - f(a_0)^{\frac{1}{q} - 1} a_m^{\prime} = (1 + f(a_0(x))^{\frac{1}{q} - 1})c_m - \lambda(x). \]

\noindent As $a_0(x) \not\in k_2$ we have $f(a_0(x))^{\frac{1}{q} - 1} \neq 0,-1$ and hence the fiber of $\pi$ over $x$ is isomorphic (over $\bk$) to the Artin-Schreier covering of $\bA^1(c_m)$, hence is itself isomorphic to the affine line. This shows the first statement of the lemma. 

For the second statement, note that as the fibers of $\pi$ are $\cong \bA^1$, we have $\R_c^2\pi_{\ast}\overline{\bQ_{\ell}} \cong \overline{\bQ_{\ell}}(-1)$ and $\R_c^j\pi_{\ast}\overline{\bQ_{\ell}} = 0$ for $j \neq 2$. This together with the spectral sequence 

\[ \coh_c^i(Y_{v,0}^{m-1,\prime} \sm Z^{m-1}, \R_c^j\pi_{\ast}\overline{\bQ_{\ell}}) \Rar \coh_c^{i+j}(Y_{v,0}^{m,\prime} \sm Z^m) \] 

\noindent implies the second statement of part (i). Part (ii) of the lemma has a similar proof, using \eqref{eq:first_coord_change_for_Zm2} instead of \eqref{eq:defining_Ym_over_Ym-1}.
\end{proof}

\begin{lm}\label{lm:induction_on_Z_tilde_Z}
With notations as in the proof of Theorem \ref{thm:coh_of_Yvm}, for $m \geq 3$, we have $\tilde{Z}^{m-1} \cap H \cong \coprod\limits_{k_-(\alpha_1)} \coprod\limits_{N_+(\alpha_2,c_1)}  Z^{m-2}$. 
\end{lm}
\begin{proof}[Proof of Lemma \ref{lm:induction_on_Z_tilde_Z}] On the union of hyperplanes $H$, the term $(c_1 - c_1^{q^2})c_{m-1}^q$ in the equation \eqref{eq:first_coord_change_for_Zm2} defining $\tilde{Z}^{m-1}$ over $\tilde{Z}^{m-2}$ cancels, leaving the free variable $c_m$ and the equation arising from it (after renaming the variables by $\alpha_i \mapsto \alpha_{i-2}$ for $i \geq 3$, $c_i \mapsto c_{i-1}$ for $i \geq 2$) is simply the equation defining $Z^{m-2}$ over $Z^{m-3}$. The lemma follows from this observation.
\end{proof}

We recall the definition of maximal varieties from the introduction of \cite{BW}, where it appears in a similar setup. Let $X$ be a scheme of finite type over a finite field $\bF_Q$ with $Q$ elements. Let $\Frob_Q$ denote the Frobenius over $\bF_Q$. By \cite{De} Theorem 3.3.1, for each $i$ and each eigenvalue $\alpha$ of $\Frob_Q$ in $\coh_c^i(X)$, there exists an integer $i^{\prime} \leq i$, such that all complex conjugates of $\alpha$ have absolute value $Q^{i^{\prime}/2}$. Hence by Grothendieck-Lefschetz formula we get an upper bound on the number of points on $X$:

\begin{equation}\label{eq:upper_bound_via_Lefshetz} 
\sharp X(\bF_Q) = \sum_{i \in \bZ} (-1)^i \tr(\Frob_Q; \coh_c^i(X)) \leq \sum_{i \in \bZ} Q^{i/2} h_c^i(X), 
\end{equation}

\noindent where equality holds if and only if $\Frob_Q$ acts in $\coh_c^i(X)$ by the scalar $(-1)^i Q^{i/2}$ for each $i \in \bZ$. If this is the case, then $X/\bF_Q$ is called \textit{maximal}.

\begin{lm}\label{lm:Zm1_maximality}
Let $Z^m_1$ be as in the proof of Theorem \ref{thm:coh_of_Yvm}. For $m \geq 1$, $Z^m_1$ is a maximal variety over $\bF_{q^2}$.
\end{lm}

\begin{proof}
$\Frob_{q^2}$ acts on $\coh_c^i(Z^m_1)$ as an endomorphism with eigenvalues being Weil numbers with absolute value $(q^2)^{i^{\prime}/2} \leq (q^2)^{i/2} = q^i$. From \eqref{eq:upper_bound_via_Lefshetz} we obtain the upper bound $u(Z^m_1,q^2)$ for the number of $\bF_{q^2}$-points on $Z^m_1$:

\[ \sharp Z^m_1(\bF_{q^2}) \leq u(Z^m_1,q^2) = \sum_{i = 0}^{2m} q^ih_c^i(Z^m_1). \]

\noindent Using equation \eqref{eq:dim_coh_Zm}, we see that $u(Z^m_1,q^2) = q^{3m}$. On the other hand, let $p_i(\underline{c}) = \sum_{j=1}^{i-1}c_jc_{i-j}^q$ and let $c_j \in \bF_{q^2}$ for $j =1,\dots,m$ be given. Then we have $p_i(\underline{c})^q = p_i(\underline{c})$, i.e., $p_i((c_j)_{j=1}^{i-1}) \in \bF_q$ for all $1 \leq i \leq m$. But the equation $x^q + x = \lambda \in \bF_q$ has precisely $q$ solutions in $\bF_{q^2}$. Thus for each given point $(c_1,\dots,c_m) \in \bA^m(\bF_{q^2})$, there are exactly $q^m$ points in $Z^m_1$ lying over it (cf. equation \eqref{eq:first_coord_change_for_Zm1}). Thus $\sharp Z^m_1(\bF_{q^2}) = q^{3m}$, which finishes the proof.
\end{proof}


\subsection{Character subspaces} \label{sec:character_subspaces}

We keep notations from Sections \ref{sec:generalities_GL_2}-\ref{sec:coh_of_Yvm} and deduce some corollaries from Theorem \ref{thm:coh_of_Yvm}.

\begin{lm}\label{lm:IIm-action_factors_over_a_Twm_action}
Let $m \geq 0$. The $I_{m,w_m}/I^m$-action on $\coh_c^i(Y_v^m)$ factors through a $T_{w,m}$-action.
\end{lm}

\begin{proof} This is immediate as the action of $\ker(I_{m,w_m}/I^m \tar T_{w,m})$ on \\ $Y_v^m = Y_v^{m,\prime} \times \bA^{n - 1}(a_{m+1},\dots,a_{n-1}, A_0,\dots,A_{m-1})$ (where $Y_v^{m,\prime}$ is defined analogously to $Y_{v,0}^{m,\prime}$ in the proof of theorem \ref{thm:coh_of_Yvm}) comes from an action on $\bA^{n - 1}$, which contributes to the cohomology of $Y_v^m$ only via a dimension shift. \qedhere
\end{proof}

For an abelian (locally compact) group $A$, let $A^{\vee}$ denote the group of (smooth) $\overline{\bQ}_{\ell}^{\ast}$-valued characters of $A$. By Lemma \ref{lm:IIm-action_factors_over_a_Twm_action} we have a decomposition 

\begin{equation} \label{eq:dec_for_torus_chars_in_gerneral}
\coh_c^i(Y_v^m) = \bigoplus_{\chi \in T_{w,m}^{\vee}} \coh_c^i(Y_v^m)[\chi] 
\end{equation}

\noindent into isotypical components with respect to the action of $T_{w,m}$. 

\begin{cor}\label{cor:non-triv-character-spaces}
Let $m \geq 1$ and $1 \leq j \leq m$. Let $\chi \colon T_{w,m} \rar \overline{\bQ}_{\ell}^{\ast}$ be a character. Then $\Frob_{q^2}$ acts in $\coh_c^{d_0 - j}(Y_v^m)[\chi]$ by multiplication with the scalar $\chi(-1)(-1)^{d_0 - j} q^{d_0 - j}$.  
\end{cor}

\begin{proof}
We have $\coh_c^i(Y_v^m) = \bigoplus_{(k[t]/t^{m+1})^{\ast}} \coh_c^i(Y_{v,0}^m)$ and $\Frob_{q^2}$ acts trivially on the index set of the direct sum, so it is enough to study its action on $\coh_c^i(Y_{v,0}^m)$. With notations as in the proof of Theorem \ref{thm:coh_of_Yvm}, we have for $1 < j \leq m$:

\[ \coh_c^{d_0 - j}(Y_{v,0}^m) = \coh_c^{2m + 1 - j}(Y^{m,\prime}_{v,0})(-(n-1)) = \coh_c^{2m + 1 - j}(Z^m)(-(n-1)) = (\bigoplus_{N_-} \overline{\bQ}_{\ell}) \otimes \coh_c^{2m + 1 - j} (Z^m_1)(-(n-1)), \] 

\noindent as $Z^m = N_- \times Z^m_1$, where $N_-$ is seen as a disjoint union of points (for $j = 1$ this remains true if we replace the second equality by an inclusion, cf. \eqref{eq:coh_of_Yvm0_prime}). Now, $\Frob_{q^2}$ acts on $N_-$ by $(a_0,c_0) \mapsto (a_0,-c_0)$ and in $\coh_c^{2m + 1 - j}(Z^m_1)$  by the scalar $(-1)^{2m + 1 - j}q^{2m + 1 - j}$. Note that $-1 \in T_{w,m}$ acts on $N_-$ in the same way as $\Frob_{q^2}$ and trivially in $\coh_c^{2m + 1 -j}(Z^m_1)$. Thus the eigenspaces for $-1$ and $\Frob_{q^2}$ coinside. There are only two such eigenspaces $U_1$ and $U_{-1}$, and $\Frob_{q^2}$ acts in $U_{\pm 1}$ by the scalar $(\pm 1) (-1)^{2m + 1 - j}q^{2m + 1 - j}$. Now let $\overline{\chi}$ be the restriction of $\chi$ to $\mu_2 \subseteq T_{w,m}$. Then 

\[ \coh_c^i(Y_v^m)[\chi] \subseteq \coh_c^i(Y_v^m)[\overline{\chi}] = U_{\overline{\chi}(-1)}, \]

\noindent which proves the corollary. \qedhere
\end{proof}

Let $T_{w,m}^i$ denote the subgroup of $T_{w,m}$ of elements which are congruent $1$ modulo $t^i$. Let $T_{w,m}^{\vee, gen}$ denote the set of all characters of $T_{w,m}$, which are non-trivial on $T_{w,m,0} \cap T_{w,m}^m$. We also need the following purity result.

\begin{cor} \label{cor:purity_of_char_spaces} 
Let $m \geq 1$. Let $d_0 = d_0(n,m)$. The finite \'etale morphism $Y_v^m \rar Y_v^{m-1} \times \bA^1(A_{m-1})$ induces an isomorphism 

\[ \coh_c^i(Y_{v,0}^m) \cong \coh_c^i(Y_{v,0}^{m-1} \times \bA^1(A_{m-1})) \cong \coh_c^{i-2}(Y_{v,0}^{m-1})(-1) \] 

\noindent for all $i \neq d_0 - m$. If $\chi \in T_{w,m}^{\vee,{\rm gen}}$, then 

\[ \coh^i_c(Y_v^m)[\chi] = 0 \quad \text{ for all $i \neq d_0 - m$.} \]

\noindent Conversely, if $\chi \in T_{w,m}^{\vee} \sm T_{w,m}^{\vee,{\rm gen}}$, then $\coh^{d_0 - m}_c(Y_v^m)[\chi] = 0$.
\end{cor}

\begin{proof} The first statement follows directly from Theorem \ref{thm:coh_of_Yvm} by comparing dimensions. Let $N = \ker((k[t]/t^{m+1})^{\ast} \tar (k[t]/t^m)^{\ast})$. The finite \'etale covering $Y_v^m \rar Y_v^{m-1} \times \bA^1(A_{m-1})$ factors as $Y_v^m \rar \coprod_N Y_v^{m-1} \times \bA^1(A_{m-1}) \rar Y_v^{m-1} \times \bA^1(A_{m-1})$, where the first morphism has Galois group $T_{w,m,0} \cap T_{w,m}^m$. The first statement of the corollary implies that the first morphism in this factorization induces an isomorphism in the cohomology for all $i \neq d_0 - m$. The second statement of the corollary follows from it. If $\chi$ is trivial on $T_{w,m,0} \cap T_{w,m}^m$, then 

\[ \coh^{d_0 - m}_c(Y_v^m)[\chi] \subseteq \coh^{d_0 - m}_c(\coprod_N Y_v^{m-1} \times \bA^1(A_{m-1})) = \bigoplus_N \coh_c^{d_0 - m - 2}(Y_v^{m-1}) = \bigoplus_N \coh_c^{d_0(n,m-1) - m}(Y_v^{m-1}) \] 

\noindent and the last group is $0$ by Theorem \ref{thm:coh_of_Yvm}. Hence the third statement of the corollary.
\end{proof}

\subsection{Superbasic case} \label{sec:superbasic_case}

Before going on, we make a digression and study the varieties $X_{x_m}^m(b)$ in the superbasic case $b = \matzz{}{1}{t}{}$. Let $\dot{x} = \matzz{}{t^{-2k}}{t^{2k+1}}{}$ and let $x$ resp. $x_m$ be the image of $\dot{x}$ in $\tilde{W}$ resp. in $D_{G,m}(x)$. Let $\dot{v}$, $v$ be as in \eqref{eq:dot_w_and_v}. The group $J_b(F)$ is the group of units $D^{\ast}$ of the quaternion algebra $D$ over $F$. If $\caO_D$ are the integers of $D$, then $U_D = \caO_D^{\ast}$ is a maximal compact subgroup of $D^{\ast}$ and $D^{\ast}/U_D \cong \bZ$. Then \cite{Iv} Theorem 3.3(i) shows 

\[ X_x(b) = \coprod_{D^{\ast}/U_D} C_v. \]

\noindent The same arguments as used in the proof of Theorem \ref{thm:structure_of_level_m_coverings_b=1}, show that for $m < 2k$ one has:

\[ X_{w_m}^m(b) = \coprod_{D^{\ast}/U_D} Y_v^m(b), \]

where $Y_v^m(b) \subseteq C_v^m$ is the closed subscheme given by equations

\begin{eqnarray*} 
D^{-1}\sigma(C)(1 - ta\sigma(a))^{-1} &=& 1 \\
 C^{-1} \sigma(D) (1 - ta\sigma(a)) &=& 1 \\
B &=& 0.
\end{eqnarray*}

\noindent Again after eliminating $D$, it is defined in the coordinates $a,C,A$ by 

\begin{equation}\label{eq:Xxb_superbasic}
C \sigma(1 - ta\sigma(a)) = (1 - ta\sigma(a))\sigma^2(C). 
\end{equation}

An explicit comparison with results of Boyarchenko \cite{Bo}, who carried out the closely related construction of Lusztig for a division algebra over $F$ of invariant $\frac{1}{n}$ (for levels $m = 1,2$, with a suggestion of how one can continue for higher levels) and Chen \cite{Ch} (who then extended Boyarchenko's results to all levels for the quaternion algebra) shows that the varieties $X_h$ defined in the quoted papers are very similar to varieties $X_{w_m}^m(b)$ defined by \eqref{eq:Xxb_superbasic}, but do not coincide completely, at least due to the presence of the 'level zero equation' $c_0^{q^2-1} = 1$ in our approach. Also note that level $h \geq 2$ in the quoted papers correspond to level $m = h - 1 \geq 1$ in the present article.


\section{Representation theory of $GL_2(F)$} \label{sec:Rep_theory_of_GL_2}

We continue to assume $G = GL_2$ throughout this section and keep the notations from Section \ref{sec:generalities_GL_2} and the beginning of Section \ref{sec:the_structure_of_Xm_wm1}. Let us collect some further important notation here. We try to keep it consistent with the notation in \cite{BH}. The only major difference is that we write $K$ (and not $U = U_{\fM}$) for the maximal compact subgroup $G(\caO_F)$ of $G(F)$. For $\lambda \in X_{\ast}(T)$ we write $t^{\lambda} \in T(L)$ for the image of the uniformizer $t$ under $\lambda$. For an element $x \in \bar{k}[t]/t^{m+1}$, we mean by its $t$-adic valuation $v_t(x)$ the least integer $\mu \geq 0$, such that $x \in t^{\mu} \cdot \bar{k}[t]/t^{m+1}$. Moreover:

\begin{itemize}
\item[$\bullet$] $Z$ is the center of $G(F)$
\item[$\bullet$] $E = k_2((t)) \subset L$ is the unramified degree two extension of $F$ 
\item[$\bullet$] $U_M$ (resp. $U_M^m$ for $m \geq 1$) denote the units (resp. the $m$-units) of a local field $M$
\item[$\bullet$] $\fM = M_{22}(\caO_F)$ is an $\caO_F$-algebra
\item[$\bullet$] $K = G(\caO_F) = \fM^{\ast}$ is a maximal compact subgroup of $G(F)$
\item[$\bullet$] $K^i = 1 + t^i \fM$ for $i \geq 0$, and $(K^i)_{i=0}^{\infty}$ defines a descending filtration of $K$ by open normal subgroups
\item[$\bullet$] $K_m = K/K^{m+1} \cong G(\caO_F/t^{m+1})$
\item[$\bullet$] $K_m^i = K^i/K^{m+1} = 1 + t^i \fM / 1 + t^{m+1} \fM$ define a filtration on $K_m$
\item[$\bullet$] $T_{w,m}, T_{w,m,0}$ are as in the beginning of Section \ref{sec:structure_of_Yvm} and $T_{w,m}^i, T_{w,m}^{\vee,gen}$ as in the paragraph preceding Corollary \ref{cor:purity_of_char_spaces}
\end{itemize}

For a \emph{locally} compact abelian group $A$, a $\overline{\bQ}_{\ell}$-vector space $W$ with a right $A$-action and a $\overline{\bQ}_{\ell}^{\ast}$-valued character $\chi$ of $A$, we let $W[\chi]$ be the maximal \emph{quotient} of $W$, on which $A$ acts by $\chi$ (if $A$ is compact, $W$ finite dimensional, $W[\chi]$ is canonically isomorphic to the maximal $\chi$-isotypical subspace of $W$). A left $G(F)$-action on $W$, which commutes with the $A$-action, induces a left $G(F)$-action on $W[\chi]$.


\subsection{Definitions and results} \label{sec:def_of_reps}

Let $\iota_E \colon E \har M_{22}(L)$ be the embedding of $F$-algebras given by $e \mapsto {\rm diag}(e,\sigma(e))$. We have $\iota_E(U_E)/\iota_E(U_E^m) = T_{w,m}$. Inflating the $T_{w,m}$-action to $\iota_E(U_E)$ and pulling back via $\iota_E$, we obtain an $U_E$-action on $X_{w_m}^m(1)$. The center $Z$ of $G(F)$ is $F^{\ast}$, thus (as in the last lines of Section \ref{sec:def_of_cov_of_ADLV}), the action of $U_E$ on $X_{w_m}^m(1)$ extends to an action of $E^{\ast} = F^{\ast}U_E$, which commutes with the left $G(F)$-action. 

Let $\chi$ be a non-trivial character of $E^{\ast}$. The \emph{level} $\ell(\chi)$ of $\chi$ is the least integer $m \geq 0$, such that $\chi|_{U_E^{m+1}}$ is trivial. Moreover, the pair $(E/F, \chi)$ is said to be \emph{admissible} (\cite{BH} 18.2) if $\chi$ do not factor through the norm $\N_{E/F} \colon E^{\ast} \rar F^{\ast}$. Two pairs $(E/F, \chi), (E/F, \chi^{\prime})$ are $F$-isomorphic, if there is some $\gamma \in \Gal_{E/F}$ such that $\chi^{\prime} = \chi \circ \gamma$. Let $\bP_2^{\nr}(F)$ denote the set of all isomorphism classes of admissible pairs over $F$ attached to $E/F$. An admissible pair $(E/F, \chi)$ is called \emph{minimal} if $\chi|_{U_E^m}$ do not factor through $\N_{E/F}$, where $m$ is the level of $\chi$. 

\begin{Def}
Let $\chi$ be a character of $E^{\ast}$, such that $(E/F,\chi)$ is admissible. The \emph{essential level} $\ell_{\rm ess}(\chi)$ of $\chi$ is the smallest integer $m^{\prime} \geq 0$, such that there is a character $\phi$ of $F^{\ast}$ with $\ell(\phi_E \chi) = m^{\prime}$, where $\phi_E = \phi \circ \N_{E/F}$.
\end{Def}

Clearly, $\ell_{\rm ess}(\chi) \leq \ell(\chi)$. Moreover, an admissible pair $(E/F,\chi)$ is minimal if and only if $\ell(\chi) = \ell_{\rm ess}(\chi)$. Using the geometric constructions from last sections, we associate to any admissible pair $(E/F,\chi)$ a $G(F)$-representation. 

\begin{Def}\label{Def:R_chi}
Let $(E/F, \chi)$ be admissible. Let $\ell(\chi) = m \geq \ell_{\rm ess}(\chi) = m^{\prime}$. We take $n > 0$ even such that $0 \leq m < n$ and let $d_0(n,m)$ be as in Theorem \ref{thm:coh_of_Yvm}. Define $R_{\chi}$ to be the $G(F)$-representation
\[ R_{\chi} = \coh_c^{d_0(n,m) - m^{\prime}}(X_{w_m}^m(1), \overline{\bQ}_{\ell})[\chi]. \]
\end{Def}

One easily sees that this definition is independent of the choice of $n$ (cf. Theorem \ref{thm:structure_of_level_m_coverings_b=1} and the definition of $Y_v^m$). To state our main result, we need some terminology from \cite{BH}, which we will freely use here. In particular, the \emph{level} $\ell(\pi) \in \frac{1}{2}\bZ$ of an irreducible $G(F)$-representation is defined in \cite{BH} 12.6. Moreover, in \cite{BH} 20.1, 20.3 Lemma it is explained when an irreducible cuspidal representation $\pi$ of $G(F)$ is called \emph{unramified}. We denote by $\cA_2^{\nr}(F)$ the set of all isomorphism classes of irreducible cuspidal unramfied representations of $G(F)$. This is a subset of the set $\cA_2^0(F)$ of the isomorphism classes of all irreducible cuspidal representations of $G(F)$ (\cite{BH}, \S20). The (unramified part of the) tame parametrization theorem (\cite{BH} 20.2 Theorem) states the existence of a certain bijection (also for even $q$): 

\begin{equation}\label{eq:tameparamthm_BH}
\bP_2^{\nr}(F) \stackrel{\sim}{\longrar} \cA_2^{\nr}(F), \quad (E/F, \chi) \mapsto \pi_{\chi}
\end{equation}

\noindent where $\pi_{\chi}$ is a certain $G(F)$-representation constructed in \cite{BH} \S19. Below, in Section \ref{sec:second_part_of_the_proof} we briefly recall this construction. Here is our main result.

\begin{thm}\label{thm:main_thm} 
Let $(E/F,\chi)$ be an admissible pair. The representation $R_{\chi}$ is irreducible cuspidal, unramified, has level $\ell(\chi)$ and central character $\chi|_{F^{\ast}}$. 
Moreover, $R_{\chi}$ is isomorphic to $\pi_{\chi}$, i.e., the map

\begin{equation} \label{eq:R_realization}
R \colon \bP_2^{\nr}(F) \rar \cA_2^{\nr}(F), \quad (E/F, \chi) \mapsto R_{\chi} 
\end{equation}

\noindent is a bijection and coincides with the map from the tame parametrization theorem \eqref{eq:tameparamthm_BH}.
\end{thm}

The theorem will be proven at the end of Section \ref{sec:second_part_of_the_proof}, after the necessary preparations in Sections \ref{sec:trace_I_preliminaries}-\ref{sec:second_part_of_the_proof} are done. We wish to point out here, that the injectivity of \eqref{eq:R_realization} follows from results of Section \ref{sec:trace_III_Hm} and do not use the theory developed in \cite{BH}, whereas to prove surjectivity of \eqref{eq:R_realization}, we use the full machinery of \cite{BH}. 

In the rest of this section, we only deal with the central character, reduce to the minimal case and introduce some further notation. For a character $\phi$ of $F^{\ast}$ and a representation $\pi$ of $G(F)$, we write $\phi \pi$ for the $G(F)$-representation given by $g \mapsto \phi(\det(g)) \pi(g)$ and we let $\phi_E = \phi \circ N_{E/F}$ be the corresponding character of $E^{\ast}$. If $\phi$ is a character of $F^{\ast}$ and $(E/F,\chi)$ an admissible pair, then \eqref{eq:tameparamthm_BH} satisfies: the central character of $\pi_{\chi}$ is $\chi|_{F^{\ast}}$ and $\phi\pi_{\chi} = \pi_{\chi\phi_E}$. We have an analogous statement for $R_{\chi}$.

\begin{lm}\label{lm:red_to_adm_pairs} Let $(E/F,\chi)$ be admissible. Then the central character of $R_{\chi}$ is $\chi|_{F^{\ast}}$. If $\phi$ is a character of $F^{\ast}$, then $\phi R_{\chi} = R_{\chi\phi_E}$.
\end{lm}

\begin{proof}
The first statement follows from the definition of $R_{\chi}$ as the $\chi$-isotypic component of some cohomology space, and the fact that the actions of $F^{\ast} \cong Z \subseteq G(F)$ and $F^{\ast} \subseteq E^{\ast}$ in this cohomology space coincide as they already coincide on the level of varieties. The second statement follows by unraveling the definition of $R_{\chi}$, using the natural isomorphism $\coh_c^{d_0(n,m) - m^{\prime}}(Y_{v,0}^m) \cong  \coh_c^{d_0(n,\lambda) - m^{\prime}}(Y_{v,0}^{\lambda})$ for $m \geq \lambda \geq m^{\prime}$ from Theorem \ref{thm:coh_of_Yvm} and $\chi|_{\ker(\N_{E/F})} = \phi_E \chi|_{\ker(\N_{E/F})}$.
\end{proof}

We fix some notations for the rest of Section \ref{sec:Rep_theory_of_GL_2}. Let $(E/F,\chi)$ be a minimal pair, let $m$ be the level of $\chi$ and write $i_0 = d_0(n,m) - m$. Let $\tilde{Y}_v^m$ be the (disjoint) union of all $Z$-translates of $Y_v^m$. Note that $Y_v^m$ is fixed by $K$, hence $\tilde{Y}_v^m$ is fixed by $ZK$. Define the $ZK$-representation $\Xi_{\chi}$  by 

\[ \Xi_{\chi} = \coh_c^{i_0}(\tilde{Y}_v^m)[\chi]. \]

\noindent Then Theorem \ref{thm:structure_of_level_m_coverings_b=1} shows

\[ R_{\chi} = \cIndd\nolimits_{ZK}^{G(F)} \Xi_{\chi}. \]

\noindent Moreover, let $\xi_{\chi}$ be the restriction of $\Xi_{\chi}$ to $K$, i.e., $\Xi_{\chi}$ is the unique extension of $\xi_{\chi}$ to $ZK$ such that $t^{(1,1)}$ acts as $\chi(t^{(1,1)})$. Let $V_{\chi}$ denote the space in which $\Xi_{\chi}$ (resp. $\xi_{\chi}$) acts. Note that $\xi_{\chi}$ is inflated from a representation of the finite group $K_m = K/K^{m+1}$, as $K^{m+1}$ acts trivially in the cohomology of $Y_v^m$.

\begin{lm}\label{lm:induction_for_admissibility}
$\xi_{\chi} \cong \coh_c^{i_0}(Y_v^m)[\chi|_{U_E}]$.
\end{lm}

\begin{proof}
Let $W$ be a $\overline{\bQ}_{\ell}$-vector space in which $K$ acts on the left, $U_E$ on the right such that these actions commute. Let 

\[ W[t] = \left\{ \sum\limits_{\substack{i \in \bZ \\ -\infty \ll i \ll \infty}} w_i t^i \colon w_i \in W \right\} \] 

\noindent with obvious $K$ and $U_E$-actions. Extend them to $ZK$- resp. $E^{\ast}$-actions by letting $t$ act as $t.(\sum_i w_i t^i) = \sum_i w_{i-1} t^i$. Then one checks that $(W[t])[\chi|_{U_E}] = (W[\chi|_{U_E}])[t]$ and that the composed map $W[\chi|_{U_E}] \har (W[\chi|_{U_E}])[t] \tar (W[t])[\chi]$ is a bijection. Apply this to $W = \coh_c^{i_0}(Y_v^m)$ and $W[t] \cong \cIndd_K^{ZK} W = \coh_c^{i_0}(\tilde{Y}_v^m)$.
\end{proof}


\subsection{Trace computations I: preliminaries} \label{sec:trace_I_preliminaries}

Let $(E/F,\chi)$ be a minimal pair of level $m \geq 0$. Via $\iota_E$, $\chi|_{U_E}$ induces a character of $T_{w,m}$. We denote this character of $T_{w,m}$ also by $\chi$. The context excludes any ambiguity.

\begin{lm}\label{lm:if_chi_minimal_then_chi_gen}
We have $\chi \in T_{w,m}^{\vee,gen}$, i.e., $\chi$ is non-trivial on $T_{w,m,0} \cap T_{w,m}^m$.
\end{lm}

\begin{proof}
As $(E/F,\chi)$ is minimal, $\chi|_{ \ker(\N_{E/F} \colon U_E^m \rar U_F^m) }$ is non-trivial. The level of $\chi$ is $m$, i.e., $\chi$ is trivial on $U_E^{m+1}$. Now we have $U_E^m/U_E^{m+1} \cong T_{w,m}^m$ via $\iota_E$, the norm map induces the map $\bar{\N} \colon U_E^m/U_E^{m+1} \rar (k[t]/t^{m+1})^{\ast}$, and moreover, $\det \circ \iota_E  = \bar{\N}$, where $\det \colon T_{w,m}^m \rar (k[t]/t^{m+1})^{\ast}$ is the determinant. Now $\ker(\det \colon T_{w,m}^m \rar (k[t]/t^{m+1})^{\ast}) = T_{w,m,0} \cap T_{w,m}^m$ and $\chi \in T_{w,m}^{\vee,gen}$ if and only it is non-trivial on $T_{w,m,0} \cap T_{w,m}^m$. This shows the lemma.
\end{proof}

In Section \ref{sec:def_of_reps} we defined the $K_m$-representation $\xi_{\chi}$. Our goal in Sections \ref{sec:trace_I_preliminaries}-\ref{sec:trace_III_Hm} will be to compute the trace of $\xi_{\chi}$ on some important subgroups of $K_m$. We will use the following trace formula from \cite{Bo}, which is similar to \cite{DL} Theorem 3.2 and is adapted to cover the situation with wild ramification.

\begin{lm}[\cite{Bo} Lemma  2.12] \label{lm:Boyarchenko_trace_formula}
Let $X$ be a separated scheme of finite type over a finite field $\bF_Q$ with $Q$ elements, on which a finite group $A$ acts on the right. Let $g \colon X \rar X$ be an automorphism of $X$, which commutes with the action of $A$. Let $\psi \colon A \rar \overline{\bQ}_{\ell}^{\ast}$ be a character of $A$. Assume that $\coh_c^i(X)[\psi] = 0$ for $i \neq i_0$ and $\Frob_Q$ acts on $\coh_c^{i_0}(X)[\psi]$ by a scalar $\lambda \in \overline{\bQ}_{\ell}^{\ast}$. Then 

\[ {\rm Tr}(g^{\ast}, \coh_c^{i_0}(X)[\psi]) = \frac{(-1)^{i_0}}{\lambda \cdot \sharp{A}} \sum_{\tau \in A} \psi(\tau) \cdot \sharp S_{g,\tau}, \]

\noindent where $S_{g,\tau} = \{ x \in X(\overline{\bF_q}) \colon g(\Frob_Q(x)) = x \cdot \tau \}$.
\end{lm}

We adapt this to our situation. Recall from \eqref{eq:Def_of_Y_m} and Proposition \ref{prop:structure_of_Yvm_bla}(i) that $Y_v^m$ was parametrized by coordinates $a \in \bk[t]/t^{n-1}, C \in (\bk[t]/t^{m+1})^{\ast}, A \in \bk[t]/t^{m+1}$ with $a_0 = a \mod t \not\in k$. We use Lemma \ref{lm:Boyarchenko_trace_formula} with $Q = q^2$ and $X = Y_v^m$.

\begin{lm}\label{lm:simplifying_the_equations_a_bit}
Let $(E/F,\chi)$ be a minimal pair and let $g \in K$. Assume $g$ acts on $Y_v^m$ such that there is some rational expression $p(g,a) \in (\bar{k}[t]/t^{m+1})^{\ast}$ in $a$, such that $g.(a,C,A) = (g.a, p(g,a) \cdot C,g.A)$ on coordinates. Let $\tau \in T_{w,m}$. Then

\[ \Tr(g^{\ast}; \coh_c^{i_0}(Y_v^m)) = \frac{1}{q^{m+1}} \sum_{\tau \in T_{w,m}} \chi(\tau) \sharp S_{g,\tau}^{\prime}, \]

\noindent where $S_{g,\tau}^{\prime}$ is the set of solutions in the variable $a = a \mod t^{m+1} \in \bar{k}[t]/t^{m+1}$ (with $a_0 \in \bk \sm k$) of the equations

\begin{eqnarray}
\sigma(\sigma(a) - a)(\sigma(a) - a)^{-1} &=& - p(\sigma^2(a),g)^{-1} \tau \label{eq:simplified_eq_for_S_xtau} \\
g.\sigma^2(a)  &=& a. \nonumber
\end{eqnarray}
\end{lm}

\begin{proof} First we claim that $\sharp S_{g,\tau} = q^{2(n-1)}(q^2 -1)q^{2m} \sharp S_{g,-\tau}^{\prime}$. One computes that $g$ acts trivially on the coordinates $A_0, \dots, A_{m-1}$ (cf. \eqref{eq:general_formula_for_action_on_Yvm}). We observe that a point $(a,c,A) \in Y_v^m(\bar{k})$ lies in $S_{g,\tau}$ if and only if it satisfies the following equations:

\begin{eqnarray*}
\sigma^2(C)C^{-1} &=& \sigma(\sigma(a) - a)(\sigma(a) - a)^{-1} \mod t^{m + 1} \\
g.\sigma^2(a)  &=& a \label{eq:for_unipotent_action_1} \\ 
p(\sigma^2(a),g) \cdot \sigma^2(C) &=& C \cdot \tau \\
\sigma^2(A) &=& A. 
\end{eqnarray*}

\noindent Here $A_0,\dots A_{m-1}$ are 'free' coordinates and contribute a factor $q^2$ each. Also $a_{m+1}, \dots, a_{n-1}$ occur only in the second equation, and it is easy to see that each of them also contributes the factor $q^2$. Finally, for a fixed $a$, the third equation has exactly $(q^2 - 1)q^{2m}$ different solutions in $C \in (\bk[t]/t^{m+1})^{\ast}$, and we can eliminate $C$ by putting the first and the third equations together. Thus $\sharp S_{g,\tau}$ is equal to $q^{2(n-1)} \cdot (q^2 - 1)q^{2m}$ times the number of solutions of equations \eqref{eq:simplified_eq_for_S_xtau} in $a \in \bk[t]/t^{m+1}$ with $\tau$ replaced by $-\tau$. This shows the claim.

Now, Corollaries \ref{cor:non-triv-character-spaces} and \ref{cor:purity_of_char_spaces} show that the conditions of Lemma \ref{lm:Boyarchenko_trace_formula} are satisfied and the lemma follows from an easy computation involving the above claim. \qedhere
\end{proof}

\begin{Def}
For $x \in G(\caO_L/t^{m+1}) \sm \{ 1 \}$ the \emph{level} of $x$ is the maximal integer $\ell(x) \geq 0$, such that $x \equiv 1 \mod t^{\ell(x)}$. The level of $1$ is $m+1$. 
\end{Def}

This definition is auxiliary and will be used in the next two sections. Note that the level is invariant under conjugation in $G(\caO_L/t^{m+1})$.


\subsection{Trace computations II: $N_m$-action on $V_{\chi}$}\label{sec:unipotent_action_and_central_char}

We keep notations from Sections \ref{sec:def_of_reps},\ref{sec:trace_I_preliminaries}. Let further $N_m \subset K_m$ denote the subgroup

\[ N_m = \left\{ \matzz{1}{u}{}{1} \colon x \in k[t]/t^{m+1} \right\}. \]

\noindent For $0 \leq i \leq m+1$, let $N_m^i$ denote the subgroup of $N_m$ consisting of elements congruent to $1$ modulo $t^i$ and let $N_m^{\vee,gen}$ denote the set of characters of $N_m$, which are non-trivial on $N_m^m$.


\begin{prop}\label{prop:V_chi_as_U_m-representation}
As $N_m$-representations one has 
\[ \xi_{\chi} = \Ind\nolimits_1^{N_m} 1 - \Ind\nolimits_{N_m^m}^{N_m} 1 = \bigoplus_{\psi \in N_m^{\vee,gen}}  \psi \] 

\noindent In particular, $\dim_{\overline{\bQ}_{\ell}} V_{\chi} = (q-1)q^m$.
\end{prop}

\begin{proof} We claim that for $g \in N_m$ we have

\begin{equation}\label{eq:traces_of_unipotents} 
{\rm Tr}(g^{\ast}, V_{\chi}) = \begin{cases} q^{m+1} - q^m & \text{if $g = 1$} \\ -q^m & \text{if $g \in N_m^m \sm \{ 1 \}$} \\ 0 & \text{if $g \not\in N_m^m$} \end{cases}. 
\end{equation}

The proposition follows from this claim by comparing the traces of the $N_m$-representations on the left and the right sides. We need the following lemma. Let $S_{g,\tau}^{\prime}$ be as in Lemma \ref{lm:simplifying_the_equations_a_bit}.

\begin{lm}\label{lm:computing_unipotent_traces} Let $g \in N_m$ of level $\ell(g) \leq m+1$. Then $S_{g,\tau}^{\prime} = \emptyset$, unless $v_t(1 - \tau) = \ell(g)$ and $\tau \cdot \sigma(\tau) = 1$ in $k_2[t]/t^{m+1}$. If both are satisfied, then
\[ \sharp S_{g,\tau}^{\prime} = 
\begin{cases} 
 (q-1)q^{2m+1} & \text{if $g = 1$ (and hence $\tau = 1$),} \\ 
q^{m + 1 + \ell(g)}  & \text{if $g \in N_m \sm \{1\}$.}
\end{cases} \]


\end{lm}

We postpone the proof of Lemma \ref{lm:computing_unipotent_traces} and finish the proof of our claim. For $g = 1$, we have $S_{1,\tau}^{\prime} = \emptyset$ unless $\tau = 1$ and $\sharp S_{1,1}^{\prime} = q^m(q-1)$. The claim follows immediately from Lemma \ref{lm:simplifying_the_equations_a_bit}. Let now $g \in N_m \sm \{1\}$ of level $\ell = \ell(g) \leq m$. Then Lemma \ref{lm:simplifying_the_equations_a_bit} shows

\begin{eqnarray*} 
{\rm Tr}(g^{\ast},  V_{\chi}) &=& \frac{1}{q^{m+1}} \sum_{\tau \in T_{w,m,0} \cap (T_{w,m}^{\ell} \sm T_{w,m}^{\ell + 1})} \chi(\tau)\sharp S_{g,\tau}^{\prime}  \\ 
&=& q^{\ell} \cdot \sum_{\tau \in T_{w,m,0} \cap (T_{w,m}^{\ell} \sm T_{w,m}^{\ell+1})} \chi(\tau) = -q^{\ell} \cdot \sum_{\tau \in T_{w,m,0} \cap  T_{w,m}^{\ell+1}} \chi(\tau),
\end{eqnarray*}

\noindent the last equation being true as $\chi$ is a non-trivial on $T_{w,m,0} \cap T_{w,m}^{\ell}$. Unless $\ell = m$, the sum in the last expression vanishes, as $\chi$ is still a non-trivial character on $T_{w,m,0} \cap  T_{w,m}^{\ell + 1}$. If $\ell = m$, we have $T_{w,m}^{\ell + 1} = \{1\}$, and \eqref{eq:traces_of_unipotents} follows. 
\end{proof}

\begin{proof}[Proof of Lemma \ref{lm:computing_unipotent_traces}]
As both, the $T_{w,m}$- and the $K_m$-actions on $Y_v^m$ have their origin in matrix multiplication, one sees easily that $S_{g,\tau} = \emptyset$, unless $\det(\tau) = \det(g)$. Thus we can assume this, i.e., $\tau\sigma(\tau) = 1$. Write $g = \matzz{1}{x}{}{1} \in N_m$ with $x \in k[t]/t^{m+1}$. Then $v_t(x) = \ell(g)$. The action of $g$ can be describen by $g.(a,C) = (a + x, C)$. By Definition, $S_{g,\tau}^{\prime}$ is the set of solutions of

\begin{eqnarray}\label{eq:for_unipotent_action_15}
\sigma(\sigma(a) - a)(\sigma(a) - a)^{-1} &=& - \tau \\
\sigma^2(a) + x &=& a \nonumber
\end{eqnarray}

\noindent in $a \in \bar{k}[t]/t^{m+1}$ with $a_0 \not\in k$. Let $s = \sigma(a) - a$. For a fixed $s \in (\bar{k}[t]/t^{m+1})^{\ast}$, the equation $\sigma(a) - a = s$ in $a$ has exactly $q^{m+1}$ solutions. After adding $-\sigma(a)$ to both sides of the second equation in \eqref{eq:for_unipotent_action_15}, this second equation gets $\sigma(s) + x = -s$.  Thus we are reduced to solve the equations

\begin{eqnarray} 
\sigma(s) + x &=& -s \label{eq:for_unipotent_action_2} \\ 
\sigma(s)s^{-1} &=& - \tau \nonumber
\end{eqnarray}

\noindent in the variable $s \in (\bk[t]/t^{m+1})^{\ast}$. Putting $\sigma(s) = -(x + s)$ into the second equation, we obtain $1 - \tau = -xs^{-1}$ and one checks that equations \eqref{eq:for_unipotent_action_2} are equivalent to 

\begin{eqnarray} 
\sigma(s) + x &=& -s \label{eq:for_unipotent_action_3} \\ 
1 - \tau &=& -xs^{-1} \nonumber
\end{eqnarray}

\noindent From the second equation of \eqref{eq:for_unipotent_action_3} and since $s$ is must be a unit, we see that either $S_{g,\tau}^{\prime} = \emptyset$ or $v_t(1 - \tau) = v_t(x)$. Assume the second holds and let $\mu = v_t(x)$. If $\mu = m+1$, then $g = 1$, $\tau = 1$ and the lemma follows. Assume now $0 < \mu < m+1$. Then $x = t^{\mu} \tilde{x}$ for some $\tilde{x} \in (\bk[t]/t^{m+1-\mu})^{\ast}$ and $\tau = 1 + \tilde{\tau}t^{\mu}$ for some $\tilde{\tau} \in (k_2[t]/t^{m+1-\mu})^{\ast}$. The condition $\tau \sigma(\tau) = 1$ is equivalent to 

\begin{equation} \label{eq:det_condition_forx_and_tau_for_unip_eq}
\tilde{\tau} + \sigma(\tilde{\tau}) + \tilde{\tau}\sigma(\tilde{\tau})t^{\mu} \equiv 0 \mod t^{m + 1 - \mu}.
\end{equation}

\noindent  The second equation of \eqref{eq:for_unipotent_action_3} is equivalent to $s \equiv \tilde{\tau}^{-1}\tilde{x} \mod t^{m+1 - \mu}$, i.e., $s$ is uniquely determined modulo $t^{m+1 - \mu}$. Moreover, if $s \equiv \tilde{\tau}^{-1}\tilde{x}  \mod t^{m+1 - \mu}$ we have

\[ \sigma(s) + x + s \equiv \sigma(\tilde{\tau}^{-1}\tilde{x}) + x + \tilde{\tau}^{-1}\tilde{x} = \tilde{x}(t^{\mu} + \sigma(\tilde{\tau})^{-1} + \tilde{\tau}^{-1}) \equiv 0 \mod t^{m+1-\mu}, \]

\noindent where the last equation follows from \eqref{eq:det_condition_forx_and_tau_for_unip_eq}. Thus $s \equiv \tilde{\tau}^{-1}\tilde{x} \mod t^{m+1 - \mu}$ is the unique solution of equation \eqref{eq:for_unipotent_action_3} modulo $t^{m+1 - \mu}$. One easily sees that over any solution of the first equation of \eqref{eq:for_unipotent_action_3} modulo $t^{\lambda}$ lie precisely $q$ solutions of it modulo $t^{\lambda + 1}$. This shows that \eqref{eq:for_unipotent_action_3} has precisely $q^{\mu}$ solutions if $\mu > 0$. It remains to handle the case $\mu = 0$, but this is done similarly to the case $\mu > 0$.

\end{proof}

\begin{cor}\label{cor:V_chi_irreducible}
$(\xi_{\chi},V_{\chi})$ is irreducible as $B(\caO_F/t^{m+1})$-representation and hence also as $K_m$-representation.
\end{cor}

\begin{proof}
For $\psi \in N_m^{\vee,gen}$, let $V_{\chi}[\psi]$ denote the $\psi$-isotypic component of $V_{\chi}$. By Proposition \ref{prop:V_chi_as_U_m-representation}, $V_{\chi}[\psi]$ is one-dimensional. For $x \in T(\caO_F/t^{m+1})$, let $\psi^x$ be the character of $N_m$ defined by $\psi^x(g) = \psi(x^{-1}gx)$. Then $x.V_{\chi}[\psi] \subseteq V_{\chi}[\psi^x]$. Let $0 \neq W \subseteq V_{\chi}$ be a $B(\caO_F/t^{m+1})$-invariant subspace. Then $W$ decomposes as the sum of its $N_m$-isotypical components $W[\psi]$ ($\psi \in N_m^{\vee,gen}$) and $W[\psi] \subseteq V_{\chi}[\psi]$. As $V_{\chi}[\psi]$ is one-dimensional, $W[\psi]$ is either $0$ or equal to $V_{\chi}[\psi]$. But as $W \neq 0$, there is a $\psi$, such that $W[\psi] = V_{\chi}[\psi]$. Note that the natural action of $T(\caO_F/t^{m+1})$ on $N_m^{\vee}$ restricts to a transitive action on $N_m^{\vee,gen}$. This transitivity implies that $W[\psi] = V_{\chi}[\psi]$ for all $\psi \in N_m^{\vee,gen}$, i.e., $W = V_{\chi}$.
\end{proof}


\subsection{Trace computations III: $H_m$-action on $V_{\chi}$} \label{sec:trace_III_Hm}

We keep notations from Sections \ref{sec:def_of_reps} and \ref{sec:trace_I_preliminaries}. Let $H_m \subset K_m$ be a non-split torus, i.e., a subgroup which is conjugate to $T_{w,m}$ inside $G(\caO_L/t^{m+1})$. Let $Z_m$ be the center of $K_m$. One has $Z_m \subseteq H_m$. We fix an isomorphism $c_s \colon T_{w,m} \stackrel{\sim}{\rar} H_m$, given by conjugation with $s \in G( k_2[[t]] )$, and let $H_m^i = c_s(T_{w,m}^i)$. Let $\tilde{\chi} = \chi \circ c_s^{-1}$ and $\tilde{\chi}^{\sigma} = \chi^{\sigma} \circ c_s^{-1}$, where $\chi^{\sigma} = \chi \circ \sigma$. 

Note that if $s^{\prime} \in G(L)$ is another matrix conjugating $T_{w,m}$ into $H_m$, then $c_s^{\prime}c_s^{-1}$ is either identity or $\sigma$. In particular, up to $\sigma$-action, $\tilde{\chi}$ do not depend on the choice of the element $s$. 


For a character $\psi \in H_m^{\vee}$, let $i(\psi) \in \{0, \dots, m+1 \}$ be the smallest integer such that $\psi$ coincides with $\tilde{\chi}$ or $\tilde{\chi}^{\sigma}$ on the subgroup $H_m^{i(\psi)}$ (in particular, $i(\psi) = 0$ if and only if $\psi = \tilde{\chi}$ or $\tilde{\chi}^{\sigma}$).

\begin{thm}\label{thm:H_m-decomposition_of_pi_chi} Let $\psi$ be a character of $H_m$. Then $\langle \psi, \xi_{\chi} \rangle_{H_m} = 0$ unless $\psi|_{Z_m} = \tilde{\chi}|_{Z_m}$. Assume $\psi|_{Z_m} = \tilde{\chi}|_{Z_m}$. Then 

\[ \langle \psi, \xi_{\chi} \rangle_{H_m} = \begin{cases} 1 & \text{if $m - i(\psi)$ odd} \\ 
0 & \text{if $m - i(\psi)$ even}. \end{cases} \]
\end{thm}

We have an immediate consequence.

\begin{cor}\label{cor:xi_chi_determines_chi} Let $m>0$. The character $\tilde{\chi}$ is up to $\sigma$-conjugacy uniquely determined by $\xi_{\chi}$ among all characters of $H_m$ as follows: it is the unique (up to $\sigma$-conjugacy) character $\psi \in H_m^{\vee}$, such that $\psi|_{Z_m}$ is the central character of $\xi_{\chi}$ and
\begin{itemize}
\item[(i)]  if $m$ is odd: $\psi$ occurs in $\xi_{\chi}$, and all characters $\psi^{\prime} \neq \psi$, which coincide with $\psi$ on $Z_m H_m^1$ do not occur in $\xi_{\chi}$.
\item[(ii)]  if $m$ is even: $\psi$ do not occur in $\xi_{\chi}$, and all characters $\psi^{\prime} \neq \psi$, which coincide with $\psi$ on $Z_m H_m^1$ occur in $\xi_{\chi}$.
\end{itemize}
\end{cor}

\begin{proof}[Proof of Theorem \ref{thm:H_m-decomposition_of_pi_chi}] If $\psi|_{Z_m} \neq \tilde{\chi}|_{Z_m}$, then $\langle \psi, \xi_{\chi} \rangle_{H_m} = 0$ by Lemma \ref{lm:red_to_adm_pairs}.
If $i(\psi) < m+1$, then $\psi$ coincides with precisely one of the characters $\tilde{\chi}, \tilde{\chi}^{\sigma}$ on $H_m^{i(\psi)}$ and do not coincide with the other even on the last filtration step $H_m^m$ (because $\chi \neq \chi^{\sigma}$ on $T_{w,m,0}^m$). To continue we need the following auxiliary notion.

\begin{Def} 
We say that $x$ \textit{maximal} if $\ell(x) \geq \ell(zx)$ for all $z \in Z_m$.
\end{Def}

\begin{lm}\label{lm:maximal_elements_characterization} 
Let $x = \matzz{x_1}{x_2}{x_3}{x_4} \in H_m \sm \{1\}$. Then $x$ maximal if and only if $v_t(x_3) = \ell(x)$.
\end{lm}
\begin{proof} Assume first $\ell = \ell(x) > 0$. Consider $\tau_x = c_s^{-1}(x) \in T_{w,m}$ and write $\tau_x = 1 + \tau_{x,\ell}t^{\ell} + \dots + \tau_{x,m} t^m$. One sees immediately that maximality of an element is invariant under conjugation, hence $x$ is maximal if and only if $\tau_x$ is, i.e., if and only if $\tau_{x,\ell} \not\in k$. A computation (using the fact that all entries of $s$ must be units) shows that $x_3 = t^{\ell}u(\tau_x - \sigma(\tau_x))$ with some unit $u \in (\bk[t]/t^{m + 1})^{\ast}$. The lemma follows in the case $\ell(x) > 0$. The case $\ell(x) = 0$ is similar.
\end{proof}


\begin{prop} \label{prop:traces_of_nondiag_semisimple_elements}
Let $x \in H_m$ be maximal of level $\ell(x) \leq m$. Then 
\[ \tr(x; V_{\chi}) = (-1)^{m - \ell(x) + 1}q^{\ell(x)} (\tilde{\chi}(x) + \tilde{\chi}^{\sigma}(x)). \]
\end{prop}

We postpone the proof of Proposition \ref{prop:traces_of_nondiag_semisimple_elements} and finish the proof of Theorem \ref{thm:H_m-decomposition_of_pi_chi}. Note that for $x \in H_m, z \in Z_m$ we have $\tr(zx; V_{\chi}) = \chi(z) \tr(x; V_{\chi})$. Let $\psi$ be a character of $H_m$ with $\psi|_{Z_m} = \chi|_{Z_m}$. Note that 

\[ \{x \in H_m \colon \max_{z \in Z_m} \ell(zx) = \ell \} = \begin{cases} Z_mH_m^{\ell} \sm Z_mH_m^{\ell + 1} & \text{if $\ell \leq m$} \\ Z_m & \text{if $\ell = m+1$.} \end{cases} \] 

\noindent As $\psi|_{Z_m} = \chi|_{Z_m}$ and $\tr(z; V_{\chi}) = (q-1)q^m \chi(z)$ for $z \in Z_m$ by Lemma \ref{lm:computing_unipotent_traces}, we have 

\begin{equation}\label{eq:razdrotschka_dla_psi_0_na_podsummy}
\langle \psi, \xi_{\chi} \rangle_{H_m} = \frac{1}{(q^2 - 1)q^{2m}} \sum_{x \in H_m} \psi(x)\tr(x; V_{\chi}) = \frac{1}{(q^2 - 1)q^{2m}} ((q-1)^2q^{2m} + \sum_{\ell = 0}^m S_{\ell}),
\end{equation}

\noindent where 

\begin{eqnarray*}
S_{\ell} &=& \sum_{x \in Z_mH_m^{\ell} \sm Z_mH_m^{\ell + 1}} \psi(x) \tr(x; V_{\chi}) = (-1)^{m-\ell+1}q^{\ell} \sum_{x \in Z_mH_m^{\ell} \sm Z_mH_m^{\ell + 1}} \psi(x) (\tilde{\chi}(x) + \tilde{\chi}^{\sigma}(x)) \\ 
&=& \sharp (Z_m H_m^{\ell}) \langle \psi, \tilde{\chi}+ \tilde{\chi}^{\sigma} \rangle_{Z_mH_m^{\ell}} - \sharp (Z_m H_m^{\ell + 1}) \langle \psi, \tilde{\chi} + \tilde{\chi}^{\sigma} \rangle_{Z_mH_m^{\ell + 1}}.
\end{eqnarray*}

\noindent for $0 \leq \ell \leq m$, by Proposition \ref{prop:traces_of_nondiag_semisimple_elements}. For $\ell \geq 1$ we have $\sharp (Z_m H_m^{\ell}) = (q-1)q^{2m - \ell + 1}$, and we compute: 

\[ S_0 = \begin{cases} (-1)^{m+1} q^{2m+1}(q-1) & \text{if $i(\psi)$ = 0} \\ 
(-1)^m (q-1)q^{2m} & \text{if $i(\psi) = 1$} \\ 
0 & \text{otherwise,}
\end{cases} \]

\noindent and for $0 < \ell < m$: 
\[ S_{\ell} = \begin{cases} (-1)^{m - \ell + 1} (q-1)^2q^{2m} & \text{if $i(\psi) \leq \ell$} \\ 
(-1)^{m-\ell} (q-1) q^{2m} & \text{if $i(\psi) = \ell + 1$} \\ 
0 & \text{otherwise,}
\end{cases} \]

\noindent and 

\[ S_m = \begin{cases} (-1) q^{2m} (q-1)(q-2) & \text{if $i(\psi) \leq m$} \\ 
2(q-1)q^{2m} & \text{if $i(\psi) = m + 1$.}
\end{cases} \]

\noindent The theorem follows if we put these values into \eqref{eq:razdrotschka_dla_psi_0_na_podsummy}.\qedhere
\end{proof} 

Before proving Proposition \ref{prop:traces_of_nondiag_semisimple_elements}, we introduce the following version of the characteristic polynomial of an element $x \in K_m$. Let $\ell = \ell(x)$ be the level of $x$. Let $\tilde{x}$ be some lift of $x$ to $K$. Then the characteristic polynomial of $\tilde{x}$ can be seen as the function 

\[ p_{\tilde{x}} \colon \caO_L \rar \caO_L \quad \lambda \mapsto p_{\tilde{x}}(\lambda) = \det(\lambda \cdot {\rm Id} - x). \]

\noindent Note that $p_{\tilde{x}}(U_L^{\ell}) \subseteq t^{2\ell} \caO_L$. Let now $\lambda \in U_L^{\ell}$ and $\tilde{x}_1, \tilde{x}_2$ two lifts of $x$ to $K$. Then $p_{\tilde{x}_1}(\lambda) - p_{\tilde{x}_2}(\lambda) \in t^{m + \ell + 1} \caO_L$, i.e., $p_{\tilde{x}}(\lambda)$ modulo $t^{m + \ell + 1}$ depends only on $x$, not on the lift $\tilde{x}$. This gives a well-defined map $p_x^{\prime} \colon U_L^{\ell} \rar t^{2 \ell}\caO_L/t^{m + \ell + 1}\caO_L$. Moreover, one immediately computes that this induces the following map defined as the composition:

\begin{equation}\label{eq:def_of_Mipo_chi_x} 
p_x \colon U_L^{\ell}/U_L^{m+1} \stackrel{p_x^{\prime}}\longrar t^{2 \ell}\caO_L/t^{m + \ell + 1}\caO_L \rar \caO_L/t^{m - \ell + 1}\caO_L, 
\end{equation}

\noindent where the second arrow is multiplication by $t^{-2 \ell}$. Explicitly, if $\ell > 0$ and $x = 1 + t^{\ell}\matzz{y_1}{y_2}{y_3}{y_4}$, then

\[ p_x(1 + \tilde{\tau}t^{\ell}) = (\tilde{\tau} - y_1)(\tilde{\tau} - y_4) - y_2 y_3. \]

\begin{proof}[Proof of Proposition \ref{prop:traces_of_nondiag_semisimple_elements}]
We identify $T_{w,m}$ with $(k_2[t]/t^{m+1})^{\ast} \subseteq (\bar{k}[t]/t^{m+1})^{\ast}$ by sending $\matzz{\tau}{}{}{\sigma(\tau)}$ to $\tau$. In particular, for $x \in K_m^{\ell}$ and $\tau \in T_{w,m}^{\ell}$ we have the element $p_x(\tau) \in \bar{k}[t]/(t^{m - \ell + 1})$, where $\ell$ is the level of $x$. Let $S^{\prime}_{x,\tau}$ be as in Lemma \ref{lm:simplifying_the_equations_a_bit}. 

\begin{lm}\label{lm:ugly_sizes_of_nondiag_sets} Let $x \in H_m$ be maximal of level $\ell = \ell(x) \leq m$. Let $\tau \in T_{w,m}$. Then $S^{\prime}_{x,\tau} = \emptyset$, unless $\tau \in T_{w,m}^{\ell}$ and $\det(\tau) = \det(x)$. For $\tau \in T_{w,m}^{\ell}$ with $\det(\tau) = \det(x)$ we have:
\[ \sharp S^{\prime}_{x,\tau} = 
\begin{cases} 
q^{m+\ell} & \text{if $p_x(\tau) = 0$ and $m - \ell$ even} \\
q^{m+\ell + 1} & \text{if $p_x(\tau) = 0$ and $m - \ell$ odd} \\
0 & \text{if $v_t(p_x(\tau)) < \infty$ is odd} \\
(q+1) q^{m+\ell} & \text{if $v_t(p_x(\tau)) < \infty$ is even.}
\end{cases} 
\]
\end{lm}

We postpone the proof of Lemma \ref{lm:ugly_sizes_of_nondiag_sets} and finish the proof of Proposition \ref{prop:traces_of_nondiag_semisimple_elements}. Let $x \in H_m$ be maximal of level $\ell \leq m$ and let $\tau_x = c_s^{-1}(x) \in T_{w,m}^{\ell}$. For $j^{\prime} \in \{0, 1, \dots, m-\ell, \infty \}$, let 

\begin{eqnarray*} 
T_x(j^{\prime}) &=& \{ \tau \in \tau_x T_{w,m,0}^{\ell} \cup \sigma(\tau_x)T_{w,m,0}^{\ell} \colon \tau \equiv \tau_x \text{ or } \sigma(\tau_x) \mod t^{\ell + j^{\prime} } \} \\
 &=& \tau_x \ker(T_{w,m,0}^{\ell} \tar T_{w,\ell+j^{\prime} - 1,0}^{\ell} ) \cup \sigma(\tau_x) \ker(T_{w,m,0}^{\ell} \tar T_{w,\ell+j^{\prime} - 1,0}^{\ell} ) \subseteq T_{w,m}^{\ell} 
\end{eqnarray*}

\noindent be the union of the two $\ker(T_{w,m,0}^{\ell} \tar T_{w,\ell+j^{\prime} - 1,0}^{\ell} )$-cosets inside $T_{w,m}^{\ell}$ in which $\tau_x$ and $\sigma(\tau_x)$ lie (note that these cosets are disjoint if $j^{\prime} > 0$ and equal if $j^{\prime} = 0$). Note that $\tau \in T_x(0)$ if and only if $\det(\tau) = \det(x)$ and $\tau \in T_{w,m}^{\ell}$.

\begin{sublm}\label{sublm:relation_vtA_Txj} For $\tau \in T_{w,m}^{\ell}$ with $\det(\tau) = \det(x)$ we have:
$\tau \in T_x(j^{\prime}) \LRar v_t(p_x(\tau)) \geq j^{\prime}$. 
\end{sublm}

\begin{proof}[Proof of Sublemma \ref{sublm:relation_vtA_Txj}]
Write $\tau = 1 + t^{\ell}\tilde{\tau}$ and $\tau_x = c_s^{-1}(x) = 1 + t^{\ell}\tilde{\tau}_x$. The characteristic polynomial is invariant under conjugation, hence $v_t(p_x(\tau)) = v_t(p_{\tau_x}(\tau))$. Write $\tau_x = 1 + t^{\ell}\tilde{\tau}_x$. As $x$ (and hence also $\tau_x$) is maximal, $\tilde{\tau}_x - \sigma(\tilde{\tau}_x)$ is a unit. We have $p_{\tau_x}(\tau) = (\tilde{\tau} - \tilde{\tau}_x)(\tilde{\tau} - \sigma(\tilde{\tau}_x))$. Thus $v_t(p_x(\tau)) \geq j^{\prime} \LRar \tilde{\tau} \equiv \tilde{\tau}_x \mod t^{j^{\prime}}$ or $\tilde{\tau} \equiv \sigma(\tilde{\tau}_x) \mod t^{j^{\prime}}$. The sublemma follows.
\end{proof}

By Lemma \ref{lm:simplifying_the_equations_a_bit} and the first statement of Lemma \ref{lm:ugly_sizes_of_nondiag_sets} we have:

\begin{eqnarray} 
\tr(x; V_{\chi}) &=& \frac{1}{q^{m+1}} \sum_{\substack{\tau \in T_{w,m}^{\ell} \\ \det(\tau) = \det(x)}} \chi(\tau) \sharp S_{x,\tau}^{\prime} \label{eq:tr_x_semisimple_nondiag} \\ 
&=& \frac{1}{q^{m+1}} (\sum_{\tau \in T_x(\infty)} \chi(\tau) \sharp S_{x,\tau}^{\prime} + \sum_{j^{\prime} = 0}^{m-\ell} \sum_{\substack{\tau \in T_x(j^{\prime}) \\ \tau \not\in T_x(j^{\prime}+ 1) } } \chi(\tau) \sharp S_{x,\tau}^{\prime}). \nonumber
\end{eqnarray}

\noindent Write $T_{w,m,0}^{\ell} = T_{w,m,0} \cap T_{w,m}^{\ell}$. Lemma \ref{lm:ugly_sizes_of_nondiag_sets} implies for $0 < j^{\prime} < m-\ell$: 

\begin{eqnarray*} 
\sum_{\substack{\tau \in T_x(j^{\prime}) \\ \tau \not\in T_x(j^{\prime}+ 1) } } \chi(\tau) \sharp S_{x,\tau}^{\prime} &=& (\tilde{\chi}(x) + \tilde{\chi}^{\sigma}(x)) \cdot (const) \cdot \sum_{\substack{ \tau \in \ker(T_{w,m,0}^{\ell} \tar T_{w,l+j^{\prime} - 1,0}^{\ell} ) \\  \tau \not\in \ker(T_{w,m,0}^{\ell} \tar T_{w,l+j^{\prime} ,0}^{\ell} ) } } \chi(\tau) \\
&=&  (\tilde{\chi}(x) + \tilde{\chi}^{\sigma}(x)) \cdot (const) \cdot \sum_{\tau \in \ker(T_{w,m,0}^{\ell} \tar T_{w,l+j^{\prime} ,0}^{\ell} )} \chi(\tau) = 0
\end{eqnarray*}

\noindent and similarly $\sum_{\substack{\tau \in T_x(0) \\ \tau \not\in T_x(1) } } \chi(\tau) \sharp S_{x,\tau}^{\prime} = 0$ as $\chi$ is non-trivial on $\ker(T_{w,m,0}^{\ell} \tar T_{w,l+j^{\prime} ,0})$ (one has to apply this twice). Further, if $m - \ell$ is odd, then $S_{x,\tau}^{\prime}$ is empty for $\tau \in T_x(m-\ell) \sm T_x(\infty)$, hence in this case Lemma \ref{lm:ugly_sizes_of_nondiag_sets} implies:

\[ \tr(x; V_{\chi}) = \frac{1}{q^{m+1}} \sum_{\tau \in T_x(\infty)} \chi(\tau) \sharp S_{x,\tau}^{\prime} = q^{\ell} (\chi(\tau_x) + \chi^{\sigma}(\tau_x)). \]

\noindent If $m - \ell$ is even, then 

\begin{eqnarray*} 
\tr(x; V_{\chi}) &=& \frac{1}{q^{m+1}} \left( \sum_{\tau \in T_x(\infty)} \chi(\tau) q^{m + \ell} + (\chi(\tau_x) + \chi^{\sigma}(\tau_x)) \sum_{\tau \in  T_{w,m,0}^m \sm \{1\} } \chi(\tau) (q+1)q^{m+\ell} \right) \\ 
&=& -q^{\ell}(\chi(\tau_x) + \chi^{\sigma}(\tau_x)). 
\end{eqnarray*}

\noindent This finishes the proof of Proposition \ref{prop:traces_of_nondiag_semisimple_elements}. \qedhere
\end{proof} 


\begin{proof}[Proof of Lemma \ref{lm:ugly_sizes_of_nondiag_sets}] Let $x \in H_m$ be maximal of level $\ell \leq m$ and let $\tau_x = c_s^{-1}(x)$. Let $\tau \in T_{w,m}$. From the definition of $S_{x,\tau}^{\prime}$ one immediately deduces that $S_{x,\tau}^{\prime} = \emptyset$, unless $\det(x) = \det(\tau)$, i.e., $\tau \in \tau_x T_{w,m,0} = T_x(0)$. Hence we can assume $\tau \sigma(\tau) = \det(x)$. Write $x = \matzz{x_1}{x_2}{x_3}{x_4}$. 
A point of $Y_v^m$ is parametrized by the coordinates $a,C$ and $A$ as above. One computes:

\begin{equation} \label{eq:general_formula_for_action_on_Yvm}
x.(a,C,A) = (x.a,x.C,x.A) = (\frac{x_1a + x_2}{x_3a + x_4}, \frac{\det(x)C}{x_3a + x_4}, A). 
\end{equation}

\noindent By Lemma \ref{lm:simplifying_the_equations_a_bit}, $\sharp S_{x,\tau}^{\prime}$ is the number of solutions of equations \eqref{eq:simplified_eq_for_S_xtau} in the variable $a \in \bk[t]/t^{m+1}$ (satisfying $a_0 \not\in k$). Explicitly, these equations are:
\begin{eqnarray*} 
x_1 \sigma^2(a) + x_2 &=& a(x_3 a \sigma^2(a) + x_4) \\
(\sigma^2(a) - \sigma(a)) \sigma(\tau) &=& -(x_3 \sigma^2(a) + x_4)(\sigma(a) - a).
\end{eqnarray*}

\noindent Inserting the first equation into the second and applying $\sigma^{-1}$ to the result, we see that the equations are equivalent to

\begin{eqnarray}
x_3 a \sigma^2(a) - x_1 \sigma^2(a) + x_4 a - x_2 &=& 0 \label{eq:channel_nr_3} \\
x_3 a \sigma(a) + (\tau - x_1) \sigma(a) - (\tau - x_4) a - x_2 &=& 0. \label{eq:channel_nr_4}
\end{eqnarray}

\begin{sublm}\label{sublm:eine_Gl_hat_richtig_viele_Lsgn}
For $i \geq 1$, there are precisely $q^2$ solutions of equation \eqref{eq:channel_nr_3} modulo $t^{i+1}$ lying over a given solution (satisfying $a_0 \not\in k$) of \eqref{eq:channel_nr_3} modulo $t^i$.
\end{sublm}

\begin{proof}
Write $a = \sum_{j=0}^i a_j t^j$, $x_{\lambda} = \sum_{j=0}^i x_{\lambda j}$. The coefficient of $t^i$ on the right side of \eqref{eq:channel_nr_3} modulo $t^{i+1}$ is

\begin{equation}\label{eq:um_die_anzahl_der_lsgn_nur_einer_gl_zu_erhalten} (x_{30} a_0 - x_{10})a_i^{q^2} + (x_{30} a_0^{q^2} + x_{40})a_i + R, \end{equation}

\noindent where $R \in \bk$ depends only on $a_0,\dots, a_{i-1}$ and $x$ and not on $a_i$. As $a_0 \not\in k$ and $x \in G(k)$, it is clear that $x_{30} a_0 - x_{10} \neq 0$ and $x_{30} a_0^{q^2} + x_{40} \neq 0$. Thus \eqref{eq:um_die_anzahl_der_lsgn_nur_einer_gl_zu_erhalten} is a separable polynomial in $a_i$ of degree $q^2$, i.e., it has exactly $q^2$ different roots.
\end{proof}

\noindent Now we concentrate on the case $\ell > 0$, i.e., $x = 1 + t^{\ell}\matzz{y_1}{y_2}{y_3}{y_4}$. Equation \eqref{eq:channel_nr_4} modulo $t^{\ell}$ shows $(\tau - 1) \sigma(a) = (\tau - 1) a \mod t^{\ell}$. If $\tau \not\equiv 1 \mod t^{\ell}$, then this forces $a_0 \in k$, which contradicts $(a,c,A) \in Y_v^m$. Hence $S_{x,\tau}^{\prime} = \emptyset$, unless $\tau \in T_{w,m}^{\ell}$. Assume $\tau \in T_{w,m}^{\ell}$. and $\tau = 1 + \tilde{\tau}t^{\ell}$, with some $\tilde{\tau} \in k_2[t]/t^{m+1-\ell}$. Note that the condition $\det(x) = \det(\tau)$ satisfied by $x,\tau$ is equivalent to

\begin{equation}\label{eq:det_condition_forx_and_tau} 
y_1 + y_4 + (y_1y_4 - y_2y_3)t^{\ell} \equiv \tilde{\tau} + \sigma(\tilde{\tau}) + \tilde{\tau}\sigma(\tilde{\tau})t^{\ell} \mod t^{m - \ell +1}.  
\end{equation}

\noindent Equations \eqref{eq:channel_nr_3} and \eqref{eq:channel_nr_4} transform to

\begin{eqnarray}
t^{\ell}(y_3 a \sigma^2(a) - y_1 \sigma^2(a) + y_4 a - y_2) &=& \sigma^2(a) - a \label{eq:channel_nr_3strich} \\
y_3 a \sigma(a) + (\tilde{\tau} - y_1) \sigma(a) - (\tilde{\tau} - y_4) a - y_2 &\equiv& 0 \mod t^{m-\ell + 1} \label{eq:channel_nr_4strich}
\end{eqnarray}

\noindent Sublemma \ref{sublm:eine_Gl_hat_richtig_viele_Lsgn} shows that the number of solutions of \eqref{eq:channel_nr_3strich}, \eqref{eq:channel_nr_4strich} is equal to $q^{2\ell}$ times the number of solutions of \eqref{eq:channel_nr_3strich} and \eqref{eq:channel_nr_4strich} $\mod t^{m-\ell + 1}$.

Let us write $Q = p_x(\tau)$ with $p_x$ as in \eqref{eq:def_of_Mipo_chi_x}. A computation involving \eqref{eq:det_condition_forx_and_tau} implies
\begin{equation} \label{eq:B_A_Vergleichsformel} 
\tilde{\tau} + \sigma(\tilde{\tau}) - y_1 - y_4 \equiv t^{\ell}\tau^{-1}Q \mod t^{m-\ell + 1}.
\end{equation}

\noindent Sublemma \ref{lm:maximal_elements_characterization} allows us to make the linear change of variables $a = b - \frac{\tilde{\tau} - y_1}{y_3}$ and equations \eqref{eq:channel_nr_3strich}, \eqref{eq:channel_nr_4strich} modulo $t^{m-\ell + 1}$ take the following form (using \eqref{eq:B_A_Vergleichsformel} and the fact that $\sigma^2(a) - a = \sigma^2(b) - b$):

\begin{eqnarray}
\quad t^{\ell}\left( y_3 b \sigma^2(b) - \tilde{\tau} \sigma^2(b)  - (t^{\ell} \tau^{-1} Q - \sigma(\tilde{\tau}) ) b + y_3^{-1}Q \right) &\equiv& \sigma^2(b) - b \mod t^{m-\ell + 1} \label{eq:channel_nr_3strich_mit_b} \\
y_3 b \sigma(b) - t^{\ell}\tau^{-1}Q b + y_3^{-1} Q &\equiv& 0 \mod t^{m-\ell + 1}. \label{eq:channel_nr_4strich_mit_b}
\end{eqnarray}

\noindent Write $b = \sum_{i=0}^m b_i t^i$. We have three cases: $v_t(Q) = \infty$,  $v_t(Q) < \infty$ odd, $v_t(Q) < \infty$ even. Assume first $v_t(Q) = \infty$, i.e., $Q = 0$.  Then \eqref{eq:channel_nr_4strich_mit_b} is equivalent to $b_0 = b_1 = \dots = b_{ \lfloor \frac{m - \ell}{2} \rfloor } = 0$. As $b = 0$ is also a solution of \eqref{eq:channel_nr_3strich_mit_b} $\mod t^{\lfloor \frac{m - \ell}{2} \rfloor + 1}$, it follows from Sublemma \ref{lm:maximal_elements_characterization} that the number of solutions of \eqref{eq:channel_nr_3strich_mit_b} and \eqref{eq:channel_nr_4strich_mit_b} $\mod t^{m - \ell + 1}$ is exactly $(q^2)^{m - \ell - \lfloor \frac{m - \ell}{2} \rfloor}$ and the lemma follows in this case, once we have shown that no of these solutions lies in the 'forbidden' subset, determined by $a_0 \in k$. This is done in Sublemma \ref{sublm:solutions_have_a_0_not_in_k} below.

Now assume $v_t(Q) < \infty$.  Equation \eqref{eq:channel_nr_4strich_mit_b} shows that we must have $v_t(Q) = 2v_t(b)$. In particular, $\sharp S_{x,\tau}^{\prime} = \emptyset$ if $v_t(Q)$ is odd. Assume $v_t(Q) = 2j < \infty$ is even and write $Q = t^{2j} Q^{\prime}$. Then $b \in \bar{k}[t]/t^{m - \ell + 1}$ solves \eqref{eq:channel_nr_4strich_mit_b} if and only if $b = t^j b^{\prime}$ (i.e., $b_0 = \dots = b_{j-1} = 0$) and $b^{\prime} = \sum_{i = j}^{m - \ell + 1} b_i t^{i - j}$ solves 

\begin{equation} \label{eq:channel_nr_4strich_mit_b_prime}
y_3 b^{\prime} \sigma(b^{\prime}) - t^{\ell + j}\tau^{-1}Q^{\prime} b^{\prime} + y_3^{-1} Q^{\prime} \equiv 0 \mod t^{m - \ell - 2j + 1}. 
\end{equation}

\noindent Note that such a solution $b^{\prime}$ is necessarily a unit. Using this, we can express $\sigma(b^{\prime})$ in terms of $b^{\prime}$, apply $\sigma$ to it, and then insert again the expression of $\sigma(b^{\prime})$ in \eqref{eq:channel_nr_4strich_mit_b_prime}. This shows:

\[ \sigma^2(b^{\prime}) \equiv \frac{\sigma(Q^{\prime})}{ \frac{ Q^{\prime} }{ b^{\prime} } - y_3 t^{\ell+j}\tau^{-1}Q^{\prime} } + y_3^{-1} t^{\ell + j} \sigma(\tau^{-1}Q^{\prime}) \, (\mkern-18mu\mod t^{m - \ell - 2j + 1}), \]

\noindent which multiplied by $t^j$ gives an expression of $\sigma^2(b) \mod t^{m - \ell - j + 1}$ in terms of $b^{\prime}$. Now a (very ugly, but straightforward) computation shows that if we put this expression for $\sigma^2(b)$ into equation \eqref{eq:channel_nr_3strich_mit_b} modulo $t^{m - \ell - j + 1}$, we obtain the tautological equation $0=0$. This simply means that any solution $b$ of \eqref{eq:channel_nr_4strich_mit_b} $\mod t^{m - \ell - j + 1}$ is a solution of \eqref{eq:channel_nr_3strich_mit_b}  $\mod t^{m - \ell - j + 1}$. Similarly as in Sublemma \ref{sublm:eine_Gl_hat_richtig_viele_Lsgn}, one checks that \eqref{eq:channel_nr_4strich_mit_b} modulo $t^{m - \ell - j + 1}$ has precisely $(q+1)q^{m - \ell - 2j}$ solutions ($q+1$ corresponds to the freedom of choosing $b_j$ and $q^{m - \ell - 2j}$ corresponds to the freedom of choosing $b_{j+1}, \dots, b_{m - \ell - j}$). Again by Sublemma \ref{lm:maximal_elements_characterization}, the lemma also follows in this case, once we have shown that no 
of these solutions lie in the 'forbidden' subset, 
determined by $a_0 \in 
k$. This is done in Sublemma \ref{sublm:solutions_have_a_0_not_in_k}.

In the case $\ell(x) = 0$, the lemma can be proven in the same way.
\end{proof}

\begin{sublm} \label{sublm:solutions_have_a_0_not_in_k}
With notations as in the proof of Lemma \ref{lm:ugly_sizes_of_nondiag_sets}, assume $\tau \in T_{w,m}^{\ell}$ and $\det(\tau) = \det(x)$. Let $a$ be a solution of equations \eqref{eq:channel_nr_3}, \eqref{eq:channel_nr_4}, then $a_0 \not\in k$. 
\end{sublm}

\begin{proof}[Proof of Sublemma \ref{sublm:solutions_have_a_0_not_in_k}]
For any $r \geq 1$ and an element $X \in \bar{k}[t]/(t^r)$, denote by $X_0 \in \bar{k}$ the reduction of $X$ modulo $t$. Write $\tau_x = c_s^{-1}(x) \in T_{w,m}$. We handle the case $\ell > 0$ first. Write $\tau = 1 + \tilde{\tau}t^{\ell}$, $\tau_x = 1 + \tau_{x,\ell}t^{\ell} + \dots$.  As $a$ is a solution of \eqref{eq:channel_nr_3}, \eqref{eq:channel_nr_4}, $b = a + \frac{\tilde{\tau} - y_1}{y_3}$ is a solution of \eqref{eq:channel_nr_3strich_mit_b}, \eqref{eq:channel_nr_4strich_mit_b}. We have $a_0 = b_0 - \frac{\tilde{\tau}_0 - y_{10}}{y_{30}}$. 

Assume first $v_t(Q) > 0$. Maximality of $x$ (and hence of $\tau_x$) implies $\tau_{x,\ell} \not\in k$. Now, $v_t(Q) = v_t(p_x(\tau)) > 0$ is equivalent to $\tilde{\tau}_0 \equiv \tau_{x,\ell}$ or $\equiv \sigma(\tau_{x,\ell}) \mod t$. Hence $\tilde{\tau}_0 \not\in k$. On the other hand, the solution $b$ must satisfy $b_0 = 0$ and we have $y_{10},y_{30} \in k$. As $\tilde{\tau}_0 \not\in k$ we obtain $a_0 \not\in k$. 

Now assume $v_t(Q) = 0$ and suppose that $a_0 \in k$, i.e., $a_0^q = a_0$. Then for $b_0$ we must have:

\begin{equation}\label{eq:eq_for_b_0_expressing_a_0_ratl} 
b_0^q = b_0 + \frac{\tilde{\tau}_0^q - \tilde{\tau}_0}{y_{30}}. 
\end{equation}

\noindent Putting this into equation \eqref{eq:channel_nr_4strich_mit_b} $\mod t$, we deduce that $b_0$ must satisfy

\begin{equation}\label{eq:eq_for_b_0_to_show_nonrat} 
b_0^2 + \frac{\tilde{\tau}_0^q - \tilde{\tau}_0}{y_{30}}b_0 + \frac{Q_0}{y_{30}^2} = 0,
\end{equation}

\noindent where $Q_0 = (\tilde{\tau}_0 - \tau_{x,\ell})(\tilde{\tau}_0 - \sigma(\tau_{x,\ell}))$. By assumption we have $\det(\tau_x) = \det(x) = \det(\tau) \mod t^{\ell + 1}$, hence 

\begin{equation} \label{eq:det_cond_ausgeschrieben} 
\tilde{\tau}_0 + \sigma(\tilde{\tau}_0) = \tau_{x,\ell} + \sigma(\tau_{x,\ell}). 
\end{equation}

Assume first $\charac(k) > 2$. A computation shows that the discriminant of equation \eqref{eq:eq_for_b_0_to_show_nonrat} is $D = y_{30}^{-2} (\sigma(\tau_{x,\ell}) - \tau_{x,\ell})^2$ and hence the solutions of it are

\[ b_{0,\pm} = - \frac{\sigma(\tilde{\tau}_0) - \tilde{\tau}_0}{2y_{30}} \pm \frac{\sigma(\tau_{x,\ell}) - \tau_{x,\ell}}{2y_{30}}. \]

\noindent Putting any of this solutions into equation \eqref{eq:eq_for_b_0_expressing_a_0_ratl} shows $\tau_{x,\ell} = \sigma(\tau_{x,\ell})$, which is a contradiction to maximality of $x$. This finishes the proof in the case $\charac(k) > 2$.

Assume now $\charac(k) = 2$. Let $\mu = \frac{\tilde{\tau}_0^q + \tilde{\tau}_0}{y_{30}}$. Then $\mu \in k$. Further, \eqref{eq:det_cond_ausgeschrieben} shows $\mu \neq 0$ (otherwise, $\tau_{x,\ell} \in k$, which is a contradiction to maximality of $x$). Set also $\delta = \frac{Q_0}{y_{30}^2 \mu^2}$. Note that by \eqref{eq:det_cond_ausgeschrieben}, $Q_0 \in k$ and hence also $\delta \in k$.  Make the change of variables $b_0 = \mu s$, i.e., $b_0$ satisfies \eqref{eq:eq_for_b_0_to_show_nonrat}, \eqref{eq:eq_for_b_0_expressing_a_0_ratl} if and only if $s$ satisfies 
\begin{eqnarray*}
s^q + s + 1 &=& 0 \\
s^2 + s + \delta &=& 0.
\end{eqnarray*}

\noindent The second of these equations implies $s^q = s + \Tr_{k/\bF_2}(\delta)$. This together with the first equation implies $\Tr_{k/\bF_2}(\delta) = 1$. On the other hand, let $R = \frac{\tilde{\tau}_0 + \tau_{x, \ell}}{\tau_{x, \ell} + \tau_{x, \ell}^q}$. Using \eqref{eq:det_cond_ausgeschrieben}, we see that 
\[ R + R^2 = \frac{(\tilde{\tau}_0 + \tau_{x, \ell})(\tilde{\tau}_0 + \tau_{x, \ell}^q)}{(\tau_{x, \ell} + \tau_{x, \ell}^q)^2} = \delta. \]

\noindent This implies $\Tr_{k/\bF_2}(\delta) = 0$, which is a contradiction. This proves the lemma in the case $\charac(k) = 2$. 

Let us now consider the case $\ell(x) = 0$. By maximality of $x$, we have $x_3 \in (k[t]/t^r)^{\ast}$. If $\charac(k) > 2$, the variable change $a = b - x_3^{-1}(\tau - x_1)$ analogous to the one in the case $\ell(x) > 0$ leads to a very similar proof. Assume $\charac(k) = 2$. We have to show that equations \eqref{eq:channel_nr_4} $\mod t$ and $a_0^q = a_0$ do not have a common solution. Let 

\[ \lambda = x_{30}^{-1}(x_{10} + x_{40}) = x_{30}^{-1}(\tau_{x,0}^q + \tau_{x,0}), \]

\noindent Then $\lambda^q = \lambda$ and $\lambda \neq 0$ by maximality of $x$. Make the change of variables given by $a_0 = \lambda r + x_{30}^{-1}x_{10}$. Then $a_0^q = a_0$ transforms into $r^q = r$ and \eqref{eq:channel_nr_4} gets (after using $r^q = r$ and cancelling)

\[ x_{30}\lambda^2 r^2 + x_{30}\lambda^2 r + x_{30}^{-1} \det(x) = 0, \]

\noindent or equivalently

\[ r^2 + r + \frac{\tau_{x,0}^{q+1}}{(\tau_{x,0}^q + \tau_{x,0})^2} = 0. \]

\noindent We have to show that this equation has no solution in $k$. But observe that the two solutions of it are given by $\frac{\tau_{x,0}}{\tau_{x,0}^q + \tau_{x,0}}$ and $\frac{\tau_{x,0}^q}{\tau_{x,0}^q + \tau_{x,0}}$ lie in $k_2 \sm k$ (note that they are different by maximality of $x$). This finishes the proof also in this case.
\qedhere

\end{proof}


\subsection{Relation to strata} \label{sec:first_part_of_the_proof}

We will freely use the terminology of intertwining from \cite{BH} \S 11 and of strata and cuspidal inducing data from \cite{BH} Chapter 4. From results of Section \ref{sec:unipotent_action_and_central_char} we deduce that $R_{\chi}$ is irreducible, cuspidal and contains an unramified stratum. First we have the following general result.

\begin{prop}\label{prop:unram_stratum_is_contained}
Let $m \geq 0$ and let $\Xi$ be a $ZK$-representation, which restriction to $K$  is the inflation of an irreducible $K_m$-representation $\xi$, which do not contain the trivial character on $N_m^m$. Then the $G(F)$-representation $\Pi_{\Xi} = \cIndd_{ZK}^{G(F)} \Xi$ is irreducible, cuspidal and admissible. If $m > 0$, it contains an unramified simple stratum $(\fM,m,\alpha)$ for some $\alpha \in t^{-m}\fM$. Moreover, $\ell(\Pi_{\Xi}) = m$ and $\Pi_{\Xi}$ do not contain an essentially scalar stratum. In particular, for any character $\phi$ of $F^{\ast}$, one has $0 < \ell(\Pi_{\Xi}) \leq \ell(\phi \Pi_{\Xi})$.
\end{prop}

\begin{cor}\label{cor:minimality_and_cont_an_unram_stratum}
Let $(E/F,\chi)$ be a minimal pair, such that $\chi$ has level $m \geq 0$. The representation $R_{\chi}$ is irreducible, cuspidal and admissible. Assume $m > 0$. Then the representation $R_{\chi}$ contains an unramified simple stratum. In particular, $\ell(R_{\chi}) = m$ and $R_{\chi}$ is unramified. Moreover, for any character $\phi$ of $F^{\ast}$, one has $0 < \ell(R_{\chi}) \leq \ell(\phi R_{\chi})$.
\end{cor}

\begin{proof} All assumptions of Proposition \ref{prop:unram_stratum_is_contained} are satisfied for the $ZK$-representation $\Xi_{\chi}$ and the corresponding $K_m$-representation $\xi_{\chi}$ by Corollary \ref{cor:V_chi_irreducible} and Proposition \ref{prop:V_chi_as_U_m-representation}. 
\end{proof}

\begin{proof}[Proof of Proposition \ref{prop:unram_stratum_is_contained}] Irreducibility and cuspidality of $\Pi_{\Xi}$ follow from \cite{BH} Theorem 11.4, which assumptions are satisfied due to irreducibility of $\Xi$ and Lemma \ref{lm:inter_spaces_comp}. Then admissibility follows from irreducibility (cf. e.g. \cite{BH} 10.2 Corollary). Now assume $m > 0$. To contain a stratum is a priori defined with respect to a choice of an additive character. So fix some $\psi \in F^{\vee}$ of level $1$ (i.e., $\psi|_{\caO_F}$ non-trivial, $\psi|_{t\caO_F}$ trivial). Then \cite{BH} 12.5 Proposition gives us an isomorphism (here we use $m > 0$):

\[ t^{-m}\fM/t^{-m+1}\fM \stackrel{\sim}{\longrar} (K^m/K^{m+1})^{\vee} = (K_m^m)^{\vee}, \quad a + t^{-m+1}\fM \mapsto \psi_a|_{K^m}, \]

\noindent where $\psi_a$ is given by $\psi_a(x) = \psi(\tr_{\fM}(a(x-1)))$, where $\tr_{\fM}$ is the trace map $\fM \rar \caO_F$. Explicitly, if $a = t^{-m}\matzz{a_1}{a_2}{a_3}{a_4} \in t^{-m}\fM$ and $x = 1 + t^m\matzz{x_1}{x_2}{x_3}{x_4} \in K_m^m$, then 

\begin{equation}\label{eq:explicitly_psi_fM_a} 
\psi_a(x) = \psi(a_1 x_1 + a_2 x_3 + a_3 x_2 + a_4 x_4). 
\end{equation}

We show that $\Pi_{\Xi}$ contains an unramified simple stratum. Therefore, note that $\Pi_{\Xi}$ contains the inflation to $K$ of the $K_m$-representation $\xi$. Thus it is enough to show that for any $\alpha \in t^{-m}\fM$, such that $\psi_{\alpha}$ is contained in $\xi$ on $K_m^m$, the stratum $(\fM,m,\alpha)$ is unramified simple. As in \cite{BH} 13.2, for $\alpha \in t^{-m}\fM$ we can write $\alpha = t^{-m}\alpha_0$ with $\alpha_0 \in \fM$ and let $f_{\alpha}(T) \in \caO_F[T]$ be the characteristic polynomial of $\alpha_0$. Let $\tilde{f}_{\alpha}(T)$ be its reduction modulo $t$. By definition, $(\fM,m,\alpha)$ is unramified simple if and only if $\tilde{f}_{\alpha}(T)$ is irreducible in $k[T]$, or equivalently, if and only if $\alpha_0 \mod t \in G(k)$ is not triangularizable. 


Let now $\alpha = t^{-m}\alpha_0$ be arbitrary such that $\xi$ contains $\psi_{\alpha}$ on $K_m^m$. It is enough to show that $\alpha_0 \mod t$ is not triangularizable. Suppose it is. Then there is some $\beta = t^{-m}\matzz{\beta_1}{\beta_2}{0}{\beta_4} \in t^{-m}\fM$ such that $g \alpha g^{-1} \equiv \beta \mod t^{-m+1}\fM$, i.e., $\psi_{g\alpha g^{-1}}$ and $\psi_{\beta}$ coincide on $K_m^m$. By Lemma \ref{lm:conjugacy_do_not_change_occurence}, $\psi_{\beta}$ also occurs in $\xi$ on $K_m^m$ and \eqref{eq:explicitly_psi_fM_a} shows that $\psi_{\beta}|_{N_m^m}$ is the trivial character of $N_m^m$. This is a contradiction to our assumption that $\xi$ do not contain the trivial character on $N_m^m$. This contradiction shows that $\Pi_{\Xi}$ contains an unramified simple stratum. As an unramified simple stratum is fundamental, \cite{BH} 12.9 Theorem shows that $\ell(\Pi_{\Xi}) = m$. 

Suppose now $(\fM,m^{\prime}, \alpha^{\prime})$ is some essentially scalar stratum contained in $\Pi_{\Xi}$. It has to intertwine with the previously found unramified simple stratum $(\fM, m, \alpha)$ contained in $\Pi_{\Xi}$ (cf. \cite{BH} 12.9). As essentially scalar strata are fundamental, \cite{BH} 12.9 Lemma 2 implies $m^{\prime} = m$. But in this case the above argumentation shows that $(\fM, m, \alpha^{\prime})$ is unramified simple and hence not essentially scalar. Finally, Theorem \cite{BH} 13.3 implies the last statement of the proposition.
\end{proof}

\begin{lm}\label{lm:inter_spaces_comp}
Let $\Xi$ be a $ZK$-representation, which restriction to $K$  is an inflation of a $K_m$-representation $\xi$, which do not contain the trivial character on $N_m^m$. An element $g \in G(F)$ intertwines $\Xi$ if and only if $g \in ZK$.
\end{lm}

\begin{proof}
The property of intertwining only depend on the double coset $ZKgZK$ of $g$. By Cartan decomposition, a set of representatives of these cosets is given by the diagonal matrices $\{ m_{\alpha} = t^{(0,\alpha)} \colon \alpha \in \bZ_{\geq 0} \}$ (cf. e.g. \cite{BH} 7.2.2). Obviously, $m_0 = 1$ intertwines $\Xi$, so it is enough to check that if $\alpha > 0$, then $m_{\alpha}$ do not intertwine $\Xi$. Assume $\alpha > 0$. Then $ZK \cap {}^{m_{\alpha}}(ZK) = ZK \cap m_{\alpha}ZK m_{\alpha}^{-1}$ contains the subgroup $N^m = \matzz{1}{t^m\caO_F}{}{1}$, on which $\Xi$ do not contain the trivial character, and on the other hand we have

\[ {}^{m_{\alpha}}\Xi \matzz{1}{g}{}{1} = \Xi(m_{\alpha}^{-1} \matzz{1}{g}{}{1} m_{\alpha}) = \Xi\matzz{1}{t^{\alpha}g}{}{1}, \]

\noindent i.e., ${}^{m_{\alpha}}\Xi$ restricted to $N^m$ is the trivial representation (as $\alpha > 0$ and $\Xi$ is trivial on $K^{m+1}$). Hence

\[ \Hom_{ZK \cap {}^{m_{\alpha}} (ZK)}(\Xi,{}^{m_{\alpha}} \Xi) \subseteq \Hom_{N^m}(\Xi,{}^{m_{\alpha}} \Xi) = 0. \qedhere \]
\end{proof}

\begin{lm}\label{lm:conjugacy_do_not_change_occurence}
Let $a \in t^{-m}\fM$, $g \in K$. If $\psi_{a}$ occurs in $\xi$ on $K_m^m$, then $\psi_{gag^{-1}}$ occurs in $\xi$ on $K_m^m$.
\end{lm}
\begin{proof} For $x \in K_m^m$ one has:
\[ \psi_{gag^{-1}}(x) = \psi(\tr_{\fM}(gag^{-1}(x-1))) = \psi(\tr_{\fM}(ag^{-1}(x-1)g )) = \psi(\tr_{\fM}( a(g^{-1}xg - 1) )) = \psi_a(g^{-1}xg). \]

Let $V$ denote the space in which $\xi$ acts. For simplicity we write $x.v$ instead of $\xi(x)(v)$ for $x \in K_m, v \in V$. Let $v \in V$, such that $x.v = \psi_a(x)v$ for all $x \in K_m^m$. Then for all $x\in K_m^m$ we have:
\[ g^{-1}x.(g.v) = g^{-1}xg.v = \psi_a(g^{-1}xg)v = \psi_{gag^{-1}}(x)v. \]

\noindent Thus $x.(g.v) = \psi_{gag^{-1}}(x)(g.v)$, i.e., on the linear span of $g.v$ any element $x \in K_m^m$ acts as the scalar $\psi_{gag^{-1}}(x)$. In particular, $\psi_{gag^{-1}}$ occurs in $\xi$ on $K_m^m$.
\end{proof}


\subsection{Relation to cuspidal inducing data} \label{sec:second_part_of_the_proof}

Now we want to compare our construction to the construction in \cite{BH} \S19 of representations attached to minimal pairs. For the convenience of the reader and to have appropriate notations, we briefly recall their set up (\cite{BH} \S15,\S19). Let $\psi$ be some fixed (additive) character of $F$ of level one. Let $\psi_E = \psi \circ \tr_{E/F}$, $\psi_{\fM} = \psi \circ \tr_{\fM}$. Let  $(E/F, \chi)$ be a minimal pair. Let $m > 0$ be the level of $\chi$. Let $\alpha \in \fp_E^{-m}$ be such that $\chi(1 + x) = \psi_E(\alpha x)$ for $x \in \fp_E^{\lfloor \frac{m}{2}\rfloor + 1}$. Choose an $F$-embedding  $E \har M_{22}(F)$ such that $E^{\ast} \subseteq ZK$ (not to be confused with $\iota_E$ from the beginning of Section \ref{sec:def_of_reps}). Then $(\fM, m, \alpha)$ is an unramified simple stratum. Let then 

\[ J_{\alpha} = E^{\ast}K^{\lfloor \frac{m + 1}{2} \rfloor } \]

\noindent be an open subgroup of $ZK$. Moreover, via the embedding of $E$ into $M_{22}(F)$, $\alpha$ defines (\cite{BH} 12.5) a character $\psi_{\alpha}$ of $K^{\lfloor \frac{m}{2} \rfloor + 1}$, which is trivial on $K^{m+1}$ (thus inducing a character of $K_m^{\lfloor \frac{m}{2} \rfloor + 1}$). Let $C(\psi_{\alpha}, \fM)$ be the set of isomorphism classes of all irreducible representations $\Lambda$ of $J_{\alpha}$, such that $\Lambda$ contains the character $\psi_{\alpha}$ on $K^{\lfloor \frac{m}{2} \rfloor + 1}$, or equivalently (by \cite{BH} 15.3 Theorem), $\Lambda|_{K^{\lfloor \frac{m}{2} \rfloor + 1}}$ is a multiple of $\psi_{\alpha}$. 

For any $\Lambda \in C(\psi_{\alpha}, \fM)$, the triple $(\fM, J_{\alpha}, \Lambda)$ is a (in our case, unramified) \emph{cuspidal type} in $G(F)$ in the sense of \cite{BH} 15.5 Definition. An equivalent reformulation is given in terms of \emph{cuspidal inducing data} (\cite{BH} 15.8): the cuspidal inducing datum attached to $(\fM, J_{\alpha}, \Lambda)$ is the pair $(\fM, \Xi)$, where $\Xi = \Indd_{J_{\alpha}}^{ZK} \Lambda$. The $G(F)$-representation $\cIndd_{J_{\alpha}}^{G(F)} \Lambda =\cIndd_{ZK}^{G(F)} \Xi$ attached to $(\fM, J_{\alpha}, \Lambda)$, resp. to $(\fM, \Xi)$ is then irreducible and cuspidal. 

Out of the given minimal pair $(E/F,\chi)$ one constructs now the representation $\Lambda$ of $J_{\alpha}$, and thus gets a corresponding cuspidal type. We have two different cases.

\textbf{Case $m$ odd.} (\cite{BH} 19.3) Then $\lfloor \frac{m}{2} \rfloor + 1 = \lfloor \frac{m + 1}{2} \rfloor$. Let $\Lambda$ be the character of $J_{\alpha}$ defined by
\begin{equation}\label{eq:Lambda_for_BH_constr_m_odd} 
\Lambda|_{K^{\lfloor \frac{m+1}{2} \rfloor} } = \psi_{\alpha}, \quad \Lambda|_{E^{\ast}} = \chi
\end{equation}

\noindent (this is a consistent definition, as one sees from $\tr_{\fM}|_{\caO_E} = \tr_E|_{\caO_E}$, $E^{\ast} \cap K^{\lfloor \frac{m+1}{2} \rfloor} = U_E^{\lfloor \frac{m+1}{2} \rfloor}$). 

\textbf{Case $m>0$ even.} (\cite{BH} 15.6, 19.4) Let $J_{\alpha}^1 = J_{\alpha} \cap K^1 = U_E^1 K^{\lfloor \frac{m+1}{2} \rfloor}$, $H_{\alpha}^1 = U_E^1 K^{\lfloor \frac{m}{2}\rfloor + 1}$. Then $J_{\alpha}^1 \supsetneq H_{\alpha}^1$. Let $\theta$ be the character of $H_{\alpha}^1$ defined (as in the odd case) by

\[ \theta(ux) = \chi(u)\psi_{\alpha}(x), \quad u \in U_E^1, x \in K^{\lfloor \frac{m}{2} \rfloor +1}. \]

\noindent Let $\eta$ be the unique irreducible ($q$-dimensional) $J_{\alpha}^1$-representation containing $\theta$. Let $\mu_M$ denote the group of roots unity of a field $M$ and let $\tilde{\eta}$ be the unique irreducible representation of $\mu_E/\mu_F \ltimes J_{\alpha}^1$ such that $\tilde{\eta}|_{J_{\alpha}^1} \cong \eta$ and $\tr\, \tilde{\eta}(\zeta u) = - \theta(u)$ for all $u \in H_{\alpha}^1, \zeta \in \mu_E/\mu_F \sm \{1\}$. Then $\tilde{\eta}$ factors through a representation of $\mu_E/\mu_F \ltimes J_{\alpha}^1/\ker(\theta)$. Let $\nu$ be the representation of $E^{\ast} \ltimes J_{\alpha}^1/\ker(\theta)$ which arises by inflation from $\tilde{\eta}$ via the surjection induced by $E^{\ast} \tar E^{\ast}/F^{\ast} U_E^1 \cong \mu_E/\mu_F$. Let $\tilde{\chi}$ be the character of $E^{\ast} \ltimes J_{\alpha}^1/\ker(\theta)$, which is $\chi$ on $E^{\ast}$ and trivial on $J_{\alpha}^1/\ker(\theta)$. Define the $E^{\ast} \ltimes J_{\alpha}^1/\ker(\theta)$-representation $\tilde{\Lambda} = \tilde{\chi} 
\otimes \nu$. 
It factors through the surjection $E^{\ast} \ltimes J_{\alpha}^1/\ker(\theta) \tar J_{\alpha}/\ker(\theta)$, $(e,j) \mapsto ej \mod \ker(\theta)$, hence it is an inflation of a representation $\Lambda_1$ of $J_{\alpha}/\ker(\theta)$. Take $\Lambda$ to be the inflation of $\Lambda_1$ to $J_{\alpha}$. 

Let then in both cases $(\fM,\Theta_{\chi})$ be the corresponding cuspidal inducing datum, i.e., 

\begin{equation}\label{eq:BH_def_of_cusp_ind_datum} 
\Theta_{\chi} = \Indd\nolimits_{J_{\alpha}}^{ZK} \Lambda. 
\end{equation}

\noindent Thus we attached a cuspidal inducing datum to $\chi$ and now the $G(F)$-representation $\pi_{\chi}$ from \eqref{eq:tameparamthm_BH} is defined in \cite{BH} 19.4.2 as

\[ \pi_{\chi} = \cIndd\nolimits_{ZK}^{G(F)} \Theta_{\chi} = \cIndd\nolimits_{J_{\alpha}}^{G(F)} \Lambda. \]

\begin{prop}\label{prop:comparison_with_BH_tame_param} 
Let $(E/F,\chi)$ be a minimal pair. Then $R_{\chi} \cong \pi_{\chi}$.
\end{prop}

Using Proposition \ref{prop:comparison_with_BH_tame_param}, we can prove our main result.

\begin{proof}[Proof of Theorem \ref{thm:main_thm}] By Lemma \ref{lm:red_to_adm_pairs} we can assume that $(E/F, \chi)$ is minimal in the first statement of the theorem. If $\ell(\chi) = 0$, then the first statement follows essentially from \cite{Iv} Theorem 1.1(i). If $\ell(\chi) > 0$, then the first statement follows from Corollary \ref{cor:minimality_and_cont_an_unram_stratum} and the part about the central character follows from Lemma \ref{lm:red_to_adm_pairs}.

To show $R_{\chi} \cong \pi_{\chi}$ we can assume by Lemma \ref{lm:red_to_adm_pairs} (along with the fact that $\phi \pi_{\chi} = \pi_{\phi_E \chi}$) that $(E/F,\chi)$ is minimal. Then $R_{\chi} \cong \pi_{\chi}$ follows from Proposition \ref{prop:comparison_with_BH_tame_param}. Now bijectivity of \eqref{eq:R_realization} follows from bijectivity of \eqref{eq:tameparamthm_BH}.\qedhere
\end{proof}


\begin{proof}[Proof of Proposition \ref{prop:comparison_with_BH_tame_param}] Let $m$ be the level of $\chi$. If $m = 0$, the proposition follows essentially from \cite{Iv} Theorem 1.1(i) and \cite{BH} 19.1. Assume $m > 0$. The unramified representation $R_{\chi}$ is induced from the cuspidal inducing datum $(ZK,\Xi_{\chi})$. As the map \eqref{eq:tameparamthm_BH} in the tame parametrization theorem is surjective, there is some character $\chi^{\prime}$ such that $(E/F, \chi^{\prime})$ is minimal and $R_{\chi} \cong \pi_{\chi^{\prime}}$. By Corollary \ref{cor:minimality_and_cont_an_unram_stratum}, $\ell(\chi^{\prime}) = m$. One deduces $\Xi_{\chi} \cong \Theta_{\chi^{\prime}}$ (e.g. by the same reasoning as in the proof of Lemma \ref{lm:inter_spaces_comp}). We have to show that $\chi = \chi^{\prime}$ or $\chi = (\chi^{\prime})^{\sigma}$. A comparison of the central characters shows $\chi|_{F^{\ast}} = \chi^{\prime}|_{F^{\ast}}$. Thus it remains to show that $\chi|_{U_E} = \chi^{\prime}|_{U_E}$ or $\chi|_{U_E} 
= (\chi^{\prime})^{\sigma}|_{U_E}$. The $K$-representation $\Xi_{\chi}|_
K$ is inflated from the $K_m$-representation $\xi_{\chi}$. Note that the image of $U_E$ in $K_m$ is a non-split torus $H_m$, as 
considered in Theorem \ref{thm:H_m-decomposition_of_pi_chi}. Thus $\chi|_{U_E}, \chi^{\sigma}|_{U_E}$ are the unique characters among all $U_E$-characters of level $m$, which satisfy condition (i) resp. (ii) of Corollary \ref{cor:xi_chi_determines_chi} if $m$ odd resp. even. Thus it is enough to show that $\Theta_{\chi^{\prime}}|_K$ characterizes $\chi^{\prime}|_{U_E}$ in the same way. This is the content of Lemma \ref{lm:cusp_ind_datum_cont_chi_exactly_once}. 
\end{proof}

\begin{lm}\label{lm:cusp_ind_datum_cont_chi_exactly_once}
Let $\chi$ be a character of $E^{\ast}$ of level $m > 0$ such that $(E/F, \chi)$ is a minimal pair. 
\begin{itemize}
\item[(i)] If $m$ is odd, the representation $\Theta_{\chi} = \Indd_{J_{\alpha}}^{ZK} \Lambda$ (cf. \eqref{eq:BH_def_of_cusp_ind_datum}) contains the character $\chi$ on $E^{\ast}$ (exactly once) and do not contain all the characters $\chi^{\prime}$ of $E^{\ast}$, which satisfy $\chi^{\prime}|_{U_E} \neq \chi|_{U_E}$ and $\chi^{\prime}|_{F^{\ast} U_E^1} = \chi|_{F^{\ast} U_E^1}$.
\item[(ii)] If $m$ is even, the representation $\Theta_{\chi} = \Indd_{J_{\alpha}}^{ZK} \Lambda$ (cf. \eqref{eq:BH_def_of_cusp_ind_datum}) do not contain the character $\chi$ on $E^{\ast}$ and it contains all the characters $\chi^{\prime}$ of $E^{\ast}$, which satisfy $\chi^{\prime}|_{U_E} \neq \chi|_{U_E}$ and $\chi^{\prime}|_{F^{\ast} U_E^1} = \chi|_{F^{\ast} U_E^1}$.
\end{itemize}
\end{lm}


\begin{proof}
Let first $m > 0$ be arbitrary and let $\chi^{\prime}$ be a character of $E^{\ast}$, satisfying $\chi^{\prime}|_{F^{\ast} U_E^1} = \chi|_{F^{\ast} U_E^1}$. Mackey formula and Frobenius reciprocity show:

\[ \Hom_{E^{\ast}}(\chi^{\prime}, \Theta_{\chi}) = \bigoplus\limits_{g \in E^{\ast} \backslash ZK / J_{\alpha}} \Hom_{E^{\ast} \cap {}^g J_{\alpha}}(\chi^{\prime},  {}^g \Lambda ). \]

\noindent Let $g \in ZK$.  We claim that $\Hom_{E^{\ast} \cap {}^g J_{\alpha}}(\chi^{\prime},  {}^g \Lambda ) = 0$, unless $g \in J_{\alpha}$. Indeed, we have $E^{\ast} \cap {}^g J_{\alpha} \supseteq U_E^{\lfloor \frac{m}{2} \rfloor + 1}$ and $\Lambda|_{K^{\lfloor \frac{m}{2} \rfloor + 1}}$ is a multiple of $\psi_{\alpha}$, hence ${}^g\Lambda|_{K^{\lfloor \frac{m}{2} \rfloor + 1}}$ is a multiple of $\psi_{g^{-1}\alpha g}$. Moreover, $\chi^{\prime}|_{U_E^{\lfloor \frac{m}{2} \rfloor + 1 }} = \chi|_{U_E^{\lfloor \frac{m}{2} \rfloor + 1 }} = \psi_{\alpha}$. Thus if $\Hom_{E^{\ast} \cap {}^g J_{\alpha}}(\chi^{\prime},  {}^g \Lambda ) \neq 0$, then $g$ normalizes the character $\psi_{\alpha}$ of $U_E^{\lfloor \frac{m}{2} \rfloor + 1}$. Thus Proposition \ref{prop:stronger_version_of_intertwining_theorem} shows our claim.

The claim implies that $\Hom_{E^{\ast}}(\chi^{\prime}, \Theta_{\chi}) = \Hom_{E^{\ast}}(\chi^{\prime}, \Lambda)$. In particular, if $m$ is odd, we are ready, because then $\Lambda$ is one-dimensional and $\Lambda|_{E^{\ast}} = \chi$. Assume $m$ is even. By construction, $\Lambda$ arises by an inflation process from the $E^{\ast} \ltimes J_{\alpha}^1/\ker(\theta)$-representation $\tilde{\Lambda} = \tilde{\chi} \otimes \nu$, where $\tilde{\chi}$ agrees with $\chi$ on $E^{\ast}$ and is trivial on $J_{\alpha}^1/\ker(\theta)$. So, it is enough to prove the following claim: $\nu|_{E^{\ast}}$ do not contain the trivial character of $U_E$, but it contains all non-trivial characters of $U_E$, which are trivial on $U_F U_E^1$. The restriction of $\nu$ to $E^{\ast}$ is the inflation via $E^{\ast} \tar E^{\ast}/F^{\ast}U_E^1 \cong \mu_E/\mu_F$ of the restriction to $\mu_E/\mu_F$ of the $\mu_E/\mu_F \ltimes J_{\alpha}^1$-representation $\tilde{\eta}_{\theta}$. In particular, $\nu|_{U_F U_E^1}$ is trivial. Now \cite{BH} 
19.4 Proposition shows 
that $\tilde{\eta_{\theta}}|_{\mu_E/\mu_F} = \Regul_{\mu_E/\mu_F} - 1_{\mu_E/\mu_F}$, and the claim follows. 
\end{proof}

The following proposition is an improvement of a part of the Intertwining theorem \cite{BH} 15.1. Also Lemma \ref{lm:stronger_version_of_lemma_in_intertwining_thm} below improves \cite{BH} Lemma 16.2

\begin{prop}\label{prop:stronger_version_of_intertwining_theorem}
Let $g \in ZK$. Then $g$ normalizes the character $\psi_{\alpha}$ of $U_E^{\lfloor \frac{m}{2} \rfloor + 1}$ if and only if $g \in J_{\alpha}$.
\end{prop}

\begin{proof} We can assume $g \in K$. Let $X$ be the appropriate quotient of $t^{-m} \fM/ t^{-\lfloor \frac{m}{2} \rfloor} \fM$ such that the following diagram commutes

\begin{equation}
\begin{gathered}
\begin{xy}\label{diag:character_isos_diag}
\xymatrix{
t^{-m} \fM/ t^{-\lfloor \frac{m}{2} \rfloor} \fM \ar[r]^{\sim} \ar@{->>}[d] & (K_m^{\lfloor \frac{m}{2} \rfloor + 1})^{\vee} \ar@{->>}[d] \\
X \ar[r]^(.3){\sim} & (U_E^{\lfloor \frac{m}{2} \rfloor + 1}/ U_E^{m+1})^{\vee}
}
\end{xy}
\end{gathered}
\end{equation}

\noindent where the upper horizontal map is $\alpha \mapsto \psi_{\alpha}$ with $\psi_{\alpha}$ as in \cite{BH} 12.5, and the right vertical map is restriction of characters. Let $Y \subseteq \fM$ be such that $t^{-m}Y \subseteq t^{-m}\fM$ is the preimage in $t^{-m}\fM$ of the kernel of the left vertical map. Then $g$ normalizes $\psi_{\alpha}|_{U_E^{\lfloor \frac{m}{2} \rfloor + 1}}$ if and only if the images of $g^{-1}\alpha g$ and $\alpha$  in $X$ coincide, i.e., if the following equation holds true in $t^{-m}\fM$:

\[g^{-1} \alpha g \equiv \alpha \mod t^{-\lfloor \frac{m}{2} \rfloor} \fM + t^{-m}Y.  \]

\noindent Then the result follows from Lemma \ref{lm:stronger_version_of_lemma_in_intertwining_thm} applied to $k = \lfloor \frac{m+1}{2} \rfloor$.
\end{proof}

\begin{lm} \label{lm:stronger_version_of_lemma_in_intertwining_thm} Write $\alpha = t^{-m} \alpha_0$. With notations as in the proof of Proposition \ref{prop:stronger_version_of_intertwining_theorem}, for any $1 \leq k \leq \lfloor \frac{m+1}{2} \rfloor$, we have 

\begin{equation} \label{eq:congruency_mod_tkM_ttY}
g^{-1}\alpha_0 g \equiv \alpha_0 \mod t^k \fM + Y 
\end{equation}
in $\fM$ if and only of $g \in U_E + t^k \fM$.
\end{lm}

\begin{proof}
The 'if' part is immediate. To prove the other part, we use induction on $k$ (as in \cite{BH} 16.2 Lemma). Let $k \geq 2$ and assume \eqref{eq:congruency_mod_tkM_ttY}. By induction hypothesis, $g \in U_E + t^{k-1}\fM$. We can write $g = g_1 (1+ t^{k-1}g_0)$ with $g_1 \in U_E$. Thus (as $\alpha_0 \in \caO_E$) we obtain from \eqref{eq:congruency_mod_tkM_ttY}:

\[ t^{k-1} \alpha_0 g_0 \equiv t^{k-1} g_0 \alpha_0 \mod t^k \fM + Y. \]

\noindent Thus $t^{k-1} (\alpha_0 g_0 - g_0 \alpha_0) = y + t^k m \in \fM$ for some $y \in Y$, $m \in \fM$. We deduce $y = t^{k-1}y^{\prime}$ with $y^{\prime} \in \fM$ and $\alpha_0 g_0 - g_0 \alpha_0 = y^{\prime} + tm$. We claim that $y^{\prime} \in Y + t\fM$. Indeed, this claim is equivalent to $\psi_{t^{-m} y^{\prime}}|_{U_E^m/U_E^{m+1}} \equiv 1$. But for $u \in \fM$ we have:

\[ \psi_{t^{-m}y^{\prime}}(1 + t^m u) = \psi(\tr_{\fM}(y^{\prime}u)) = \psi_{t^{-m}y}(1 + t^{m - (k-1)} u) = 1, \]

\noindent where the last equality holds as long $m - (k-1) \geq \lfloor \frac{m}{2} \rfloor + 1$, or equivalently, $k \leq \lfloor \frac{m+1}{2} \rfloor$, which is satisfied by assumption of the Lemma. This shows our claim. From it we deduce $\alpha_0 g_0 \equiv g_0 \alpha_0 \mod Y + t\fM$, i.e., by induction hypothesis, $g_0 \in \caO_E + t \fM$. Thus we are reduced to the case $k = 1$. We handle this case explicitly. The result remains unaffected if we replace the embedding $j \colon E \har M_2(F)$ by a conjugate one. As all such embeddings are $G(F)$-conjugate, we can assume that $j(\caO_E) \mod t \subseteq \fM/t\fM$ is generated as a $k$-algebra by a matrix $\beta = \matzz{}{-b}{1}{-a}$ for some $a,b \in k$ such that the characteristic polynomial $T^2 + aT + b$ is irreducible in $k[T]$ (cf. e.g. \cite{BH} 5.3). Then $\alpha = t^{-m}\alpha_0$ with $\alpha_0 \mod t = x + y\beta$ for some $x,y \in k$ and $j(\caO_E) = \caO_F[\alpha_0]$. After adding and multiplying by some central elements (which do not 
affect the condition \ref{eq:congruency_mod_tkM_ttY}), we can assume that either $\charac(k) > 2$ and there is a $D \in k^{\ast} \sm k^{\ast,2}$ such that $\alpha_0 = \matzz{}{1}{D}{}$ and or that $\charac(k) = 2$ and there is a $D \in k$ such that $T^2 + T + D \in k[T]$ is irreducible and $\alpha_0 = \matzz{}{D}{1}{1}$. We have to show that if $g \in K$ and \eqref{eq:congruency_mod_tkM_ttY} holds for $g,\alpha_0$ and $k = 1$, then $g \in \caO_E + t\fM$.

Assume first $\charac(k) > 2$. The upper horizontal map in diagram \eqref{diag:character_isos_diag} induces the isomorphism

 
\[ t^{-m}\fM/t^{-m + 1} \fM \stackrel{\sim}{\longrar} (K_m^m)^{\vee}, \]

\noindent which shows that
\begin{eqnarray*} Y + t\fM / t\fM  &=& \{ t^m \beta \colon \beta \in t^{-m}\fM \text{ and } \psi(\tr_{\fM}(\beta(e - 1))) = 1 \text{ for all } e \in 1 + t^m \caO_E \}/t\fM \\
 &=& \{ \matzz{B_1}{B_2}{-B_2 D}{-B_1} \colon B_1,B_2 \in k \}
\end{eqnarray*}

\noindent (the last equality is an easy computation). Now let $g = \matzz{g_1}{g_2}{g_3}{g_4} \in G(k)$ (we can work modulo $t$). Then condition \eqref{eq:congruency_mod_tkM_ttY} translates into

\[ \frac{1}{\det(g)} \matzz{g_3g_4 - g_1g_2D}{-g_2^2D + g_4^2}{g_1^2D - g_3^2}{g_1g_2D - g_3g_4} = g^{-1} \alpha_0 g \stackrel{!}{=} \matzz{}{1}{D}{} + \matzz{B_1}{B_2}{-B_2 D}{-B_1} \]

\noindent for some $B_1,B_2 \in k$. In particular, we must have

\[
\begin{cases} 
\frac{1}{\det(g)} (g_4^2 - g_2^2 D) = 1 + B_2 \\
\frac{1}{\det(g)} (g_1^2 D - g_3^2) = (1 - B_2)D.
\end{cases}
\]

\noindent Computing $B_2$ from the first equation and inserting it in the second, gives us

\[\frac{1}{\det(g)} (g_1^2 - g_3^2 - g_2^2 D^2 + g_4^2 D) = 2D, \]

\noindent which is equivalent to

\[ D(g_1 - g_4)^2 = (g_3 - g_2D)^2. \]

\noindent If both side are non-zero, on the left side we have a non-square in $k^{\ast}$ and on the right side we have a square, which is a contradiction. Thus both sides are zero, i.e. $g_1 = g_4$, $g_3 = g_2D$, i.e., $g \in U_E \mod t$, finishing the proof in the case $\charac(k) > 2$.

Assume now $\charac(k) = 2$. Analogously to the previous case we deduce

\[Y + t\fM/t\fM = \left\{  \matzz{B_1}{B_1 + B_3 D}{B_3}{B_1} \colon B_1,B_3 \in k \right\}.  \] 

A similar computation as above implies that for $g \in G(k)$ satisfying condition \eqref{eq:congruency_mod_tkM_ttY} we must have

\begin{eqnarray*}
\det(g)^{-1} (g_1g_2 + g_2g_3 + g_3g_4D) &=& B_1 \\
\det(g)^{-1} (g_1^2 + g_1g_3 + g_3^2 D) &=& 1 + B_3 \\
\det(g)^{-1} (g_2^2 + g_2g_4 + g_4^2 D) &=& B_1 + B_3 D + D.
\end{eqnarray*}

Putting the first and the second equation into the third and bringing some terms together shows

\[ g_2^2 + g_3^2 D^2 = (g_1^2 + g_4^2)D + g_2(g_4 + g_1 + g_3) + g_3 D(g_4 + g_1). \]

\noindent Add $2g_3^2 D = 0$ to the right side of this equation and let $A = g_2 + g_3 D$ and $B = g_1 + g_3 + g_4$. The equation is then equivalent to

\[ A^2 + B^2 D + AB = 0. \]

\noindent Suppose $B \neq 0$. Dividing by $B^2$ we obtain $(A/B)^2 + (A/B) + D = 0$, which is a contradiction to irreduciblility of $T^2 + T + D \in k[T]$, as $A/B \in k$. Thus $B = 0$ and we deduce also $A = 0$, which finishes the proof also in the case $\charac(k) = 2$.

\end{proof}


\renewcommand{\refname}{References}

\end{document}